\documentclass[reqno]{amsart}

\usepackage[foot]{amsaddr}
\usepackage{graphicx,enumerate,nicefrac,bm}
\usepackage[dvipsnames]{xcolor}
\usepackage{cite}
\usepackage{amsmath, amssymb}
\usepackage{tikz}\usetikzlibrary{matrix}
\usepackage{a4wide}
\usepackage{subfig}
\usepackage{changes}
\usepackage{algorithm,algpseudocode,tabularx}

\makeatletter
\newcommand{\multiline}[1]{%
  \begin{tabularx}{\dimexpr\linewidth-\ALG@thistlm}[t]{@{}X@{}}
    #1
  \end{tabularx}
}
\makeatother

\newcommand{\R}{\mathsf{R}}
\newcommand{\E}{\mathsf{E}}
\newcommand{\G}{\mathsf{G}}

\newcommand{\x}{\bm{x}}

\newcommand{\dx}{\,\mathsf{d}\bm{x}}

\newcommand{\HS}{\mathrm{H}^1_0(\Omega)}

\newcommand{\Lom}{\mathrm{L}^{2}(\Omega)}
\renewcommand{\P}{\omega}
\newcommand{\Pref}{\widetilde{\P}}
\newcommand{\T}{\mathcal{T}}
\newcommand{\V}{\mathbb{V}}
\newcommand{\Vh}{\widehat{\mathbb{V}}}

\newcommand{\X}{\mathbb{X}}
\renewcommand{\H}{\mathbb{H}}

\newcommand{\poly}{\mathbb{P}}
\newcommand{\jl}{[\![}
\newcommand{\jr}{]\!]}
\newcommand{\jmp}[1]{\jl#1\jr}

\newcommand{\norm}[1]{\left\|#1\right\|}

\newcommand{\dprod}[1]{\langle #1\rangle}
\DeclareMathOperator{\Span}{span}

\newcommand{\incNk}{\mathrm{inc}_N^n}
\newcommand{\dENk}{\Delta\E_N^n}

\DeclareMathOperator*{\argmin}{arg\,min}

\newtheorem{theorem}{Theorem}[section]
\newtheorem{lemma}[theorem]{Lemma}
\newtheorem{proposition}[theorem]{Proposition} 

\newtheorem{definition}[theorem]{Definition}
\theoremstyle{definition}

\newtheorem{remark}[theorem]{Remark}

\title[Gradient flow adaptive FEM for Schr\"{o}dingers equation]{Gradient flow finite element discretisations with energy-based adaptivity for excited states of Schr\"{o}dingers equation}

\author[P.~Heid]{Pascal Heid}
\address{Mathematical Institute, University of Oxford, Woodstock Road, Oxford OX2 6GG, UK}
\email{pascal.heid@maths.ox.ac.uk}

\thanks{The author acknowledges the financial support of the Swiss National Science Foundation (SNF), Project No. P2BEP2\underline{\space}191760.}

\begin{document}

\begin{abstract}
We present an effective numerical procedure, which is based on the computational scheme from [Heid et al., arXiv:1906.06954], for the numerical approximation of excited states of Schr\"{o}dingers equation. In particular, this procedure employs an adaptive interplay of gradient flow iterations and local mesh refinements, leading to a guaranteed energy decay in each step of the algorithm. The computational tests highlight that this strategy is able to provide highly accurate results, with optimal convergence rate with respect to the number of degrees of freedom.   

\end{abstract}



\maketitle

\section{Introduction}

Schr\"{o}dingers equation is the fundamental equation of physics for describing quantum mechanical behaviour, see \cite[Ch.~1]{Sakha:2020} or \cite[Ch.~1]{LinLu:2019} for an introduction to its basic theory. The time-dependent dimensionless Schr\"{o}dinger equation reads as
\begin{equation} \label{eq:timeSE}
\mathrm{i} \partial_{t} \psi(\x,t)=-\frac{1}{2} \Delta_{\x} \psi(\x,t)+V(\x) \psi(\x,t);
\end{equation}
here, $\x$ and $t$ denote the spatial and time variables, respectively, $\Delta_{\x}$ is the Laplacian in the spatial coordinates, and $\psi$ is a normalized time-dependent single-particle wavefunction. The stationary state solution of Schr\"{o}dingers equation~\eqref{eq:timeSE} can be found by solving the \emph{linear} eigenvalue problem (EVP)
\begin{equation} \label{eq:timeindSE}
-\frac{1}{2} \Delta \psi(\x) +V(\x) \psi(\x)=E \psi(\x),
\end{equation}
which is called the time-independent Schr\"{o}dinger equation. Any (normalized) eigenfunction of the EVP~\eqref{eq:timeindSE} is a quantum state of the underlying quantum mechanical system, and the corresponding eigenvalue $E$ is the total energy. Moreover, the stationary Schr\"{o}dinger equation~\eqref{eq:timeindSE} coincides with the time-independent Gross--Pitaevskii equation (GPE) for non-interacting bosons, see, e.g.,~\cite{BaoCai:2013} for a profound mathematical treatment of Bose--Einstein condensates and, in turn, of the GPE.

In this work, we restrict our focus to the following weak formulation of the time-independent Schr\"{o}dinger equation~\eqref{eq:timeindSE}: Find $E \in \mathbb{R}$ and $\psi \in \HS$ such that     
\begin{align} \label{eq:SEweak}
\int_\Omega \left(\frac{1}{2} \nabla \psi \cdot \nabla \varphi+ V(\x)\psi \varphi \right)\dx = E (\psi,\varphi)_{\Lom} \qquad \forall \varphi \in \HS,
\end{align}
where $\Omega \subset \mathbb{R}^d$, $d=\{1,2,3\}$, is a bounded, connected, and open set with Lipschitz boundary, $V\in \mathrm{L}^\infty(\Omega)$ is a potential function with $V\geq 0$ almost everywhere, and $(\cdot,\cdot)_{\Lom}$ denotes the standard $\Lom$-inner product. We further note that, upon defining the functional 
\begin{align} \label{eq:energyfunctional}
 \E(\psi):=\int_\Omega\left( \frac{1}{4} |\nabla \psi|^2 + \frac{1}{2} V(\x)|\psi|^2 \right) \dx,
\end{align}
and demanding eigenfunctions of~\eqref{eq:SEweak} to be normalized, the weak Schr\"{o}dinger equation~\eqref{eq:SEweak} is the Euler--Lagrange formulation of the constrained minimisation problem
\begin{align} \label{eq:conmin}
 \argmin_{\psi \in S_\H} \E(\psi),
\end{align}
with $S_\H:=\{\psi \in \HS: \|\psi\|_{{\rm L}^2(\Omega)}=1\}$ signifying the $\Lom$-unit sphere in~$\HS$. In particular, the weak Schr\"{o}dinger equation~\eqref{eq:SEweak} can equivalently be written as 
\begin{align*} 
\dprod{\E'(\psi),\varphi}=E (\psi,\varphi)_{{\rm L}^2(\Omega)} \qquad \forall \varphi \in \HS,
\end{align*}
with $\E'$ denoting the Fr\'{e}chet derivative and $\dprod{\cdot,\cdot}$ the duality pairing in $\mathrm{H}^{-1}(\Omega) \times \HS$. Moreover, any solution of the local minimisation problem~\eqref{eq:conmin} is an $\Lom$-normalized eigenfunction of Schr\"{o}dingers equation~\eqref{eq:SEweak}. We further note that if $\psi \in S_\H$ is an eigenfunction of \eqref{eq:SEweak} with associated eigenvalue $E$, then 
\begin{align*} 
 E=2 \E(\psi),
\end{align*}
i.e., $2 \E(\psi)$ is the energy of the quantum state $\psi$. 

Given that~$V\ge 0$ (almost everywhere in~$\Omega$), the stationary Schr\"{o}dinger equation~\eqref{eq:SEweak} has a unique ($\Lom$-normalized) positive eigenfunction~$\psi_{\mathrm{GS}}>0$, which is called the \emph{ground state}, see~\cite[Lem.~5.4]{HenningPeterseim:18}; indeed, $\psi_{\mathrm{GS}}$ is an eigenfunction to the minimal (and simple) eigenvalue, denoted by $E_{\mathrm{GS}}$, of~\eqref{eq:SEweak}, see \cite{CancesChakirMaday:10}. 

Eigenfunctions of Schr\"{o}dingers equation~\eqref{eq:SEweak} of higher (corresponding) energy are called \emph{excited states}. We emphasize that every excited state is orthogonal to the ground state, since the eigenvalue problem is linear. Indeed, it holds the following orthogonality property, which will be crucial in the analysis below.

\begin{proposition} \label{prop:orthogonal}
If $\psi_{1}$ and $\psi_{2}$ are two eigenfunctions of Schr\"{o}dingers equation~\eqref{eq:SEweak} to distinct eigenvalues $E_1$ and $E_2$, respectively, then $\psi_1$ and $\psi_2$ are orthogonal with respect to the $\Lom$-inner product, i.e.~$(\psi_1,\psi_2)_{\Lom}=0$.
\end{proposition}

\begin{proof}
First, we note that
\[\dprod{\E'(\psi_1),\psi_2}=\int_\Omega \left(\frac{1}{2} \nabla \psi_1 \cdot \nabla \psi_2+ V(\x)\psi_1 \psi_2\right)\dx=\dprod{\E'(\psi_2),\psi_1}.\]
Invoking that $\psi_1$ and $\psi_2$ are eigenfunctions of~\eqref{eq:SEweak}, this leads to
\begin{align*}
E_1 (\psi_1,\psi_2)_{\Lom}=\dprod{\E'(\psi_1),\psi_2}=\dprod{\E'(\psi_2),\psi_1}=E_2 (\psi_2,\psi_1)_{\Lom}.
\end{align*}
Then, the symmetry of the $\Lom$-inner product yields $(E_1-E_2) (\psi_1,\psi_2)_{\Lom}=0$, and thus, since $E_1 \neq E_2$, we conclude that $(\psi_1,\psi_2)_{\Lom}=0$.
\end{proof}

Finally we emphasize that every eigenspace has an orthonormal basis consisting of corresponding eigenfunctions.

\subsection*{State-of-the-art of numerical methods for Schr\"{o}dingers equation and GPE}

The stationary quantum states can either be found 
\begin{enumerate}[(i)] 
\item by solving the corresponding time-independent Schr\"{o}dinger equation~\eqref{eq:timeindSE} directly,
\item or by minimisation of the energy functional~\eqref{eq:energyfunctional} (under the $\Lom$-normalization constraint). 
\end{enumerate}

In the context of (i), classical spatial discretisation approaches such as finite element methods~\cite{bao2003ground, gong2008finite, xie2016multigrid, raza2009energy}, or finite difference schemes~\cite{chien2008two,raza2009energy} can be applied. Further approaches for solving~\eqref{eq:timeindSE} directly are, for example, given by Fourier methods~\cite{bao2003numerical, CancesChakirMaday:10}.

For the minimisation of the corresponding energy functional~\eqref{eq:energyfunctional} we mainly point to the gradient flow methods~\cite{bao2010efficient, BaoDu:04,zeng2009efficiently, HenningPeterseim:18,bao2006efficient, danaila2010new, kazemi2010minimizing, raza2009energy}. However, there is a wide range of other numerical methods in the scope of~(ii), such as the recently proposed preconditioned conjugate gradient method~\cite{antoine2017efficient} or the direct energy minimisation in the presence of symmetry properties from~\cite{bao2003ground}. We remark that most of those schemes were rather designed for the Gross--Pitaevskii equation than for Schr\"{o}dingers equation, but can also be applied for the latter. Moreover, we note that this list is far from complete.  

\subsection*{Contribution}

The aim of this work is to provide a numerical approximation procedure for the excited states of Schr\"{o}dingers equation~\eqref{eq:SEweak}, based on the simultaneous interplay of gradient flow iterations (GFI) and local adaptive finite element mesh refinements; this idea follows the recent developments on the (adaptive) \emph{iterative linearised Galerkin (ILG)} methodology~\cite{HeidWihler:20,HeidWihler:19v2,HeidWihler2:19v1,HeidStammWihler:19,HeidPraetoriusWihler:2020,CongreveWihler:17,AmreinWihler:14,AmreinWihler:15,HoustonWihler:18}, whereby adaptive discretisations and iterative nonlinearity solvers are combined in an intertwined way; we also refer to the closely related works~\cite{ErnVohralik:13,El-AlaouiErnVohralik:11,BernardiDakroubMansourSayah:15,GarauMorinZuppa:11}. 

For the purpose of our work, we will closely follow the recent paper~\cite{HeidStammWihler:19}, which proposes a computational procedure for the ground state of the Gross--Pitaevskii equation. The numerical scheme from~\cite{HeidStammWihler:19} applies an effective interplay of the gradient flow iteration method from~\cite{HenningPeterseim:18} and adaptive local mesh refinements, and thereby generates a sequence of finite element approximations defined on adaptively refined spaces which provide a corresponding sequence of monotonically decreasing energies. In the first part of this work, we will examine the gradient flow from~\cite{HenningPeterseim:18}, which was employed in the numerical scheme from~\cite{HeidStammWihler:19}, with regard to excited states of the stationary Schr\"{o}dinger equation~\eqref{eq:SEweak}. In particular, we will show that the gradient flow iteration from~\cite{HenningPeterseim:18} is not only highly suited for the ground state of the GPE, but also for excited states of Schr\"{o}dingers equation. For that purpose, we will take advantage of the orthogonality of eigenfunctions to distinct eigenvalues, which gives rise to another (linear) constraint of our optimisation problem~\eqref{eq:conmin}. Subsequently, we will use an orthogonal projection of the gradient onto the tangent space of the manifold defined by the constraints; this follows the idea in~\cite{Tanabe:1980}, or more recently in~\cite{feppon:19}. Incidentally, we will end up with the very same GFI as proposed in~\cite{HenningPeterseim:18}. However, for practical computations, a minor adjustment is presented, which, as a matter of fact, does theoretically not change the flow at all. Once the modified gradient flow is established, the adaptive interplay of gradient flow iterations and finite element discretisations is, in principal, the same as in~\cite{HeidStammWihler:19}. 

\subsection*{Outline}

In Section~\ref{sec:gradientflow} we introduce the Sobolev gradient flow and justify its use for the approximation of \emph{excited} states of Schr\"{o}dingers equation. Subsequently, Section~\ref{sec:FEM} deals with the finite element discretisation, and the adaptive interplay of the two schemes. Some numerical experiments in the two dimensional physical space are conducted in Section~\ref{sec:numerics}. Finally, we will add some concluding remarks in Section~\ref{sec:conclusion}.

\section{Sobolev gradient flow for excited states of Schr\"{o}dingers equation} \label{sec:gradientflow}

In this section, we will examine the Sobolev gradient flow from~\cite{HenningPeterseim:18} with regard to \emph{excited states} of Schr\"{o}dingers equation~\eqref{eq:SEweak}. For that purpose, we will first consider a general \emph{constrained} optimisation problem, and subsequently apply the findings in the specific context of Schr\"{o}dingers equations. 

\subsection{Dynamcial system approach for constrained minimisation problems}

Let us first consider the \emph{unconstrained} optimisation problem
\begin{align*} 
\argmin_{u \in \H} \E(u), 
\end{align*}
where $\H$ is any Hilbert space, and $\E:\H \to \mathbb{R}$ a Fr\'{e}chet differentiable functional. The dynamical system approach for the minimisation of the functional $\E$ in $\H$ is based on the gradient flow
\begin{align} \label{eq:gradientflow}
 \dot u(t)=-\nabla \E(u(t)),
\end{align}
where $\nabla \E(u) \in \H$, for fixed $u \in \H$, is the unique element in the Hilbert space $\H$ such that
\begin{align*}
 (\nabla \E(u),v)_\H=\dprod{\E'(u),v} \qquad \forall v \in \H;
\end{align*}
we emphasize that $\nabla \E \in \H$ is well-defined by the Riesz representation theorem.
The fundamental idea of this approach is that the flow should always move in direction of the steepest descent of the energy $\E$, thus leading to an energy decay along the path $u(t)$, and finally ending up in a local minimum.

Now let us consider a \emph{constrained} optimisation problem of the form
\begin{align*} 
\argmin_{u \in \H} \E(u) \ \ \text{s.t.} \ \ g(u)=0,
\end{align*}
where the constraint $g:\H \to \mathbb{R}^{m+1}$ is (componentwise) Fr\'{e}chet differentiable. We have to adapt the dynamical system~\eqref{eq:gradientflow} in such a way that $u(t)$ satisfies the constraints for all admissible $t$, i.e.~$g(u(t))=0$, given that $g(u(0))$. Following the ideas in~\cite{Tanabe:1980}, we project the gradient onto the tangent space of the set $S:=\{v \in \H: g(v)=0\} \subseteq \H$; here, the tangent space of $S$ at some point $u \in S$, denoted by $T_uS$, is given by the kernel of $g'(u):\H \to \mathbb{R}^{m+1}$, i.e., $T_uS:=\ker{\left(g'(u)\right)}:=\{v \in \H: g'(u)v =0\}$. For given $u \in S$, we denote the projection onto the tangent space by $P_{g,u}:\H \to T_uS$. Then, the modified dynamical system is given by
\begin{align*} 
 \dot u(t)=-P_{g,u(t)}(\nabla \E(u(t))).
\end{align*}
This approach was proposed by Tanabe, see~\cite{Tanabe:1980}, only for $\H=\mathbb{R}^k$ with $k \in \mathbb{N}$. We refer to~\cite{feppon:19} for the advancement of this idea for infinite dimensional Hilbert spaces; indeed, in order to define the orthogonal projection operator $P_{g,u}$ in the following, we will follow the lines in~\cite{feppon:19}. 

\begin{definition}
For any $u \in \H$, we define the transpose $g'(u)^\tau:\mathbb{R}^{m+1} \to \H$ of $g'(u):\H \to \mathbb{R}^{m+1}$ by
\begin{align} \label{eq:transpose}
 (g'(u)^\tau z,v)_\H=z \cdot g'(u)v \qquad \forall z \in \mathbb{R}^{m+1}, \ \forall v \in \H, 
\end{align}
where $z_1 \cdot z_2$ denotes the Euclidean inner product of two elements $z_1,z_2 \in \mathbb{R}^{m+1}$.
\end{definition}

\begin{remark}
 (1) For any $u \in \H$, the transpose $g'(u)^\tau:\mathbb{R}^{m+1} \to \H$ is well-defined and linear by the Riesz representation theorem.\\
 (2) If $\H=\mathbb{R}^k$ endowed with the standard Euclidean inner product, then $g'(u) \in \mathbb{R}^{m+1 \times k}$, and $g'(u)^\tau \in \mathbb{R}^{k \times m+1}$ is the usual matrix transpose.
\end{remark}

For the sake of defining the orthogonal projection operator $P_{g,u}$ we need the following \emph{qualification condition}:
\begin{enumerate}[(QC)]
 \item A point $u \in \H$ satisfies the qualification condition (QC), if $g'(u)g'(u)^\tau:\mathbb{R}^{m+1} \to \mathbb{R}^{m+1}$ is invertible. 
\end{enumerate}

Then, given that $u \in \H$ satisfies (QC), the orthogonal projection is defined by 
\begin{align} \label{eq:orthogonalprojection}
P_{g,u}:=\mathrm{I}_\H-g'(u)^\tau \left(g'(u) g'(u)^\tau\right)^{-1}g'(u),
\end{align}
where $\mathrm{I}_\H: \H \to \H$ denotes the identity mapping on $\H$. Consequently, the projected gradient flow reads as 
\begin{align} \label{eq:projectedgradientflow}
 \dot u(t)= - \nabla u(t) + g'(u(t))^\tau\left(g'(u(t))g'(u(t))^\tau\right)^{-1} g'(u(t))\nabla u(t).
 \end{align}

\begin{proposition} \label{prop:Projection}
 Let $u \in \H$ be a point satisfying the qualification condition (QC). Then the following holds:
 \begin{enumerate}[(1)]
  \item The operator $P_{g,u}: \H \to \H$ from~\eqref{eq:orthogonalprojection} is the orthogonal projection onto $T_uS=\ker{\left(g'(u)\right)}$ with respect to the $\H$-inner product. 
  \item Assuming that $u(t)$ from \eqref{eq:projectedgradientflow} is well-defined for all $t \in [0,T)$, for some $T>0$, and $g(u(0))=0$, then $g(u(t))=0$.  
 \end{enumerate}
\end{proposition}

\begin{proof}
 Ad (1): It is easily verified that $P_{g,u}$ is linear, and obviously it holds that $P_{g,u}(v)=v$ for all $v \in \ker{\left(g'(u)\right)}$. Moreover, for any $v \in \H$ and $w \in \ker{\left(g'(u)\right)}$, we have
 \begin{align*}
  (v-P_{g,u}(v),w)_\H=(v,w)_\H-(v,w)_\H+g'(u)^\tau \left(g'(u)g'(u)^\tau\right)^{-1}\underbrace{g'(u)w}_{=0}=0.
 \end{align*}
 Ad (2): Applying $g'(u(t))$ to the flow~\eqref{eq:projectedgradientflow} from the left, and noting that $g'(u(t)) \dot u(t)=\frac{d}{d t} g(u(t))$, we find that  
 \[\frac{d}{d t} g(u(t)) = -g'(u(t)) \nabla u(t)+\underbrace{g'(u(t))g'(u(t))^\tau\left(g'(u(t))g'(u(t))^\tau\right)^{-1}}_{=\mathrm{I}_{\mathbb{R}^{m+1}}} g'(u(t))\nabla u(t)=0,\]
 hence $g(u(t))$ is constant. This, together with $g(u(0))=0$, implies the claim.
\end{proof} 

\subsection{Continuous gradient flow for Schr\"{o}dingers equation} \label{sec:continuousflow}

In this subsection, we will apply the general approach from before in the context of Schr\"{o}dingers equation~\eqref{eq:SEweak}. For that purpose, we will rather consider the corresponding optimisation problem~\eqref{eq:conmin} than the linear eigenvalue problem~\eqref{eq:SEweak}.\\

Assume that some eigenfunctions $\psi_1,\dotsc,\psi_m$ of Schr\"{o}dingers equation~\eqref{eq:SEweak} with associated eigenvalues~$E_1,\dotsc,E_m$ are known; without loss of generality we may assume that those eigenvectors are orthonormal with respect to the $\Lom$-inner product. Then, in order to find a new eigenfunction of~\eqref{eq:SEweak}, we have to further restrict the constraint optimisation problem~\eqref{eq:conmin} to the orthogonal complement of the space spanned by the eigenfunctions $\psi_1,\dotsc,\psi_m$, cf.~Proposition~\ref{prop:orthogonal}. In particular, we consider the constrained optimisation problem
\begin{align} \label{eq:constrainedproblem}
\argmin_{\psi \in \HS} \E(\psi) \ \ \text{s.t.} \ \ g(\psi)=0,
\end{align}
where $\E$ is given as in~\eqref{eq:energyfunctional}, and the constraint $g:\HS \to \mathbb{R}^{m+1}$ is defined by
\begin{align} \label{eq:constraint}
g(\psi):=\left(\norm{\psi}_{\Lom}-1,(\psi_1,\psi)_{\Lom},\dotsc,(\psi_m,\psi)_{\Lom}\right)^t;
\end{align}
here, $^t:\mathbb{R}^{k \times l} \to \mathbb{R}^{l \times k}$ denotes the usual transpose operator of a real-valued matrix.

We will now employ the general approach from before to the specific constrained optimisation problem~\eqref{eq:constrainedproblem}. First of all, a straightforward calculation reveals that, for any $\psi,\varphi \in \HS$,
\begin{align} \label{eq:gderivative}
 g'(\psi)\varphi=((\psi,\varphi)_{\Lom},(\psi_1,\varphi)_{\Lom},\dotsc,(\psi_m,\varphi)_{\Lom})^t.
\end{align}
Moreover, in view of the Sobolev gradient flow from~\cite{HenningPeterseim:18} employed in~\cite{HeidStammWihler:19}, we consider the inner product 
\begin{align} \label{eq:innerproductt}
 (\psi,\varphi)_{\H}:= \int_\Omega \frac{1}{2} \nabla \psi \nabla \varphi+ V \psi \varphi \dx \qquad \forall \psi,\varphi \in \HS
\end{align}
on $\H:=\HS$. We emphasize that the space $\HS$ endowed with the inner product from~\eqref{eq:innerproductt} is indeed a Hilbert space. Additionally, it holds that
\begin{align} \label{eq:eprime}
 \dprod{\E'(\psi),\varphi}=(\psi,\varphi)_\H \qquad \forall \psi,\varphi \in \H.
\end{align}
Consequently, the gradient of $\E$ at any $\psi \in \HS$ with respect to the $\H$-inner product is simply given by 
\begin{align} \label{eq:Hgrad}
\nabla \E(\psi)=\psi. 
\end{align} 
In order to explicitly state the dynamical system ~\eqref{eq:projectedgradientflow} in the specified setting, we further define the operator $\G:\HS \to \HS$ by 
\begin{align} \label{eq:Gdef}
 (\G(\psi),\varphi)_{\H}=(\psi,\varphi)_{\Lom} \qquad \forall \psi,\varphi \in \HS,
\end{align}
which is a linear operator by the Riesz representation theorem. 

As a first step towards the explicit formulation of the gradient flow, we will determine $g'(\psi)^{\tau} \phi$, cf.~\eqref{eq:transpose}, for given $\psi \in \HS$ and $\phi=(\phi_0,\phi_1,\dotsc,\phi_m)^{t} \in \mathbb{R}^{m+1}$: By definition, it holds that 
\begin{align*}
 (g'(\psi)^t \phi,\varphi)_{\H}&\stackrel{\eqref{eq:transpose}}{=}\phi \cdot g'(\psi)\varphi \\ &\stackrel{\eqref{eq:gderivative}}{=}\phi \cdot ((\psi,\varphi)_{\Lom},(\psi_1,\varphi)_{\Lom},\dotsc,(\psi_m,\varphi)_{\Lom})^{t} \\
 &=\phi_0 (\psi,\varphi)_{\Lom}+\sum_{i=1}^m \phi_i (\psi_i,\varphi)_{\Lom}
\end{align*}
for all $\varphi \in \HS$. Consequently, it follows from~\eqref{eq:Gdef} that
\begin{align} \label{eq:gtransposeGP}
 g'(\psi)^{\tau} \phi = \phi_0 \G(\psi)+\sum_{i=1}^m \phi_i  \G(\psi_i).
\end{align}
Next, by invoking \eqref{eq:gderivative} again, a straight forward calculation reveals that
\begin{align} \label{eq:ggt1}
 g'(\psi)g'(\psi)^{\tau}=\begin{pmatrix} (\mathsf{G}(\psi),\psi)_{\Lom} & (\G(\psi_1),\psi)_{\Lom} & \dotsc & (\G(\psi_m),\psi)_{\Lom} \\ (\mathsf{G}(\psi),\psi_1)_{\Lom} & (\G(\psi_1),\psi_1)_{\Lom} & \dotsc & (\G(\psi_m),\psi_1)_{\Lom} \\ \vdots & \vdots & \ddots & \vdots \\ (\mathsf{G}(\psi),\psi_m)_{\Lom} & (\G(\psi_1),\psi_m)_{\Lom} & \dotsc & (\G(\psi_m),\psi_m)_{\Lom} \end{pmatrix},
\end{align}
for any $\psi \in \HS$. Recall that $\psi_1,\dotsc,\psi_m$ are eigenfunctions of Schr\"{o}dingers equation~\eqref{eq:SEweak} to the eigenvalues $E_1,\dotsc,E_m$. Therefore, for any $i \in \{1,\dotsc,m\}$, it holds that
\begin{align*} 
 (\psi_i,\varphi)_\H \overset{\eqref{eq:eprime}}{=} \dprod{\E'(\psi_i),\varphi} = E_i(\psi_i,\varphi)_{\Lom}\overset{\eqref{eq:Gdef}}{=}E_i(\G(\psi_i),\varphi)_\H \qquad \forall \varphi \in \HS,
\end{align*}
and thus 
\begin{align} \label{eq:Gu}
 \G(\psi_i)=\frac{\psi_i}{E_i}.
\end{align}
Since, in addition, $\{\psi_1,\dotsc,\psi_m\}$ is an $\Lom$-orthonormal set in $\HS$, the matrix from~\eqref{eq:ggt1} can be simplified:
\begin{align*} 
g'(\psi)g'(\psi)^\tau=\begin{pmatrix} (\mathsf{G}(\psi),\psi)_{\Lom} & \frac{(\psi_1,\psi)_{\Lom}}{E_1} & \dotsc & \frac{(\psi_m,\psi)_{\Lom}}{E_m} \\ (\mathsf{G}(\psi),\psi_1)_{\Lom} & \frac{1}{(E_1)^2} & \dotsc & 0 \\ \vdots & \vdots & \ddots & \vdots \\ (\mathsf{G}(\psi),\psi_m)_{\Lom} & 0 & \dotsc & \frac{1}{(E_m)^2} \end{pmatrix}.
\end{align*}
If, moreover, $\psi$ is orthogonal to $\psi_1,\dotsc,\psi_m$, then we further find that
\begin{align} \label{eq:selfadjoint}
(\mathsf{G}(\psi),\psi_i)_{\Lom}\overset{\eqref{eq:Gdef}}{=}(\mathsf{G}(\psi),\mathsf{G}(\psi_i))_\H\overset{\eqref{eq:Gdef}}{=}(\psi,\mathsf{G}(\psi_i))_{\Lom}\overset{\eqref{eq:Gu}}{=}\frac{1}{E_i}(\psi,\psi_i)_{\Lom}=0,
\end{align}
$i \in \{1,\dotsc,m\}$, and thus 
\begin{align} \label{eq:ggt3}
\left(g'(\psi)g'(\psi)^\tau\right)^{-1}=\begin{pmatrix} (\mathsf{G}\psi,\psi)^{-1}_{\Lom} & 0 & \dotsc & 0 \\ 0 & {(E_1)^2} & \dotsc & 0 \\ \vdots & \vdots & \ddots & \vdots \\ 0 & 0 & \dotsc & {(E_m)^2} \end{pmatrix}; 
\end{align}
in particular, any $\psi \in \HS$ with $g(\psi)=0$ satisfies the qualification condition (QC). This paves the way for the following result:

\begin{proposition}
Let $g,\E,\G, \H$ and $(\cdot,\cdot)_\H$ be defined as in the given Section~\ref{sec:continuousflow}. Assuming that the flow from~\eqref{eq:projectedgradientflow}, with $\psi(t)$ in place of $u(t)$, is well-defined for $t \in [0,T)$ with $g(\psi(0))=0$, then the dynamical system~\eqref{eq:projectedgradientflow} can be written in simplified form:
 \begin{align} \label{eq:simplifiedflow}
  \dot \psi(t)=-\psi(t)+\frac{1}{(\mathsf{G}(\psi(t)),\psi(t))_{\Lom}} \mathsf{G}(\psi(t)), \qquad t \in [0,T). 
\end{align}
\end{proposition}

\begin{proof}
 Since $g(\psi(0))=0$, Proposition~\ref{prop:Projection} yields that $\psi(t)$ is orthogonal to $\psi_1,\dotsc,\psi_m$ for all $t \in [0,T)$, and thus \eqref{eq:ggt3} holds true for $\psi=\psi(t)$, $t \in [0,T)$. Moreover, it follows from \eqref{eq:gderivative} that \begin{align} \label{eq:gpsi}
g'(\psi(t))\psi(t)=((\psi(t),\psi(t))_{\Lom},0,\dotsc,0)^t, \qquad t \in [0,T).
\end{align} 
Then, invoking~\eqref{eq:Hgrad},~\eqref{eq:gpsi}, \eqref{eq:ggt3}, and~\eqref{eq:gtransposeGP}, the dynamical system~\eqref{eq:projectedgradientflow} reads as
 \begin{align*}
  \dot \psi(t)&=-\psi(t)+g'(\psi(t))^{\tau} \left(\frac{(\psi(t),\psi(t))_{\Lom}}{(\mathsf{G}\psi(t),\psi(t))_{\Lom}},0,\dotsc,0\right)^t =-\psi(t)+\frac{(\psi(t),\psi(t))_{\Lom}}{(\mathsf{G}\psi(t),\psi(t))_{\Lom}} \G(\psi(t)).
 \end{align*}
Finally, as $g(\psi(t))=0$ for all $t \in [0,T)$, it especially holds that $(\psi(t),\psi(t))_{\Lom}=1$, and therefore~\eqref{eq:simplifiedflow} is established.
\end{proof}

We note that, incidentally, ~\eqref{eq:simplifiedflow} coincides with the Sobolev gradient flow from~\cite{HenningPeterseim:18}, and thus shares its features. We summarize its most important properties in the following theorem, which is basically Theorem 3.2.~from~\cite{HenningPeterseim:18}; we also refer to that reference for its proof. 

\begin{theorem} \label{thm:contflow}
Let $\psi(0) \in \HS$ such that $g(\psi(0))=0$. Then, the projected Sobolev gradient flow from~\eqref{eq:simplifiedflow} is well-defined for all $t \geq 0$, and it is energy dissipative, i.e.~$\E(\psi(t)) \leq \E(\psi(s))$ for all $0 \leq s \leq t < \infty$. Moreover, $\psi(t)$ converges in $\H$ to an $\Lom$-normalized eigenfunction $\psi^\star \in \Span\{\psi_1,\dotsc,\psi_m\}^\perp$, where $^\perp$ denotes the orthogonal complement with respect to the $\Lom$-inner product, of~\eqref{eq:SEweak} with corresponding eigenvalue
 \begin{align} \label{eq:eigenvalue}
 E^\star:=\frac{1}{(\G(\psi^\star),\psi^\star)_{\Lom}}=(\psi^\star,\psi^\star)_{\H}.
 \end{align} 
\end{theorem}

We will not prove this theorem here, but only remark that $g(\psi(t))=0$, $t \geq 0$, by Propostion~\ref{prop:Projection}, and thus $g(\psi^\star)=\lim_{t \to \infty} g(\psi(t))=0$. Consequently, $\psi^\star$ is indeed $\Lom$-normalized and orthogonal to $\psi_1,\dots,\psi_m$ with respect to the $\Lom$-inner product.

\subsection{Gradient flow iteration for Schr\"{o}dingers equation}

In order to turn the continuous gradient flow from~\eqref{eq:simplifiedflow} into a (computable) iterative scheme, we will use a forward Euler time discretisation. In particular, we define the gradient flow iteration (GFI) as in~\cite{HenningPeterseim:18,HeidStammWihler:19}. Let $\gamma^n:=\frac{1}{(\G(\psi^n),\psi^n)_{\Lom}}$ for $n \geq 0$, and let $\{\tau^n\}_{n \geq 0}$ be a sequence of positive time steps bounded from below and above, i.e.,
\begin{align*}
0<\tau_{\min} \leq \tau^n \leq \tau_{\max}< \infty \qquad \forall n \geq 0.
\end{align*} 
Then, for $\psi^0 \in \HS$ with $g(\psi^0)=0$, we define iteratively 
 \begin{align} \label{eq:GFIit}
  \widehat{\psi}^{n+1}:=(1-\tau^n)\psi^n+\tau^n \gamma^n \mathsf{G}(\psi^n) \qquad \text{and} \qquad \psi^{n+1}=\frac{\widehat{\psi}^{n+1}}{\norm{\widehat{\psi}^{n+1}}_{\Lom}}.
 \end{align}

\begin{lemma} \label{lem:constraint}
If $g(\psi^0)=0$, then $g(\psi^n)=0$ for all $n \geq 0$. 
\end{lemma}

\begin{proof}
The first constraint, $\norm{\psi^n}_{\Lom}=1$, holds obviously for all iterates $\psi^{n}$ due to the normalization step in~\eqref{eq:GFIit}. We will verify the remaining constraints, $(\psi^{n},\psi_{i})_{\Lom}=0$, for $i \in \{1,\dotsc,m\}$, by induction. By assumption it holds that $(\psi^0,\psi_i)_{\Lom}=0$ for all $i \in \{1,\dotsc,m\}$, and assume this remains true for $\psi^n$. Let $i \in \{1,\dotsc,m\}$ be arbitrary, and note that $(\psi^{n+1},\psi_i)_{\Lom}=0$ if and only if $(\widehat{\psi}^{n+1},\psi_i)_{\Lom}=0$ by the linearity of the inner product. By definition of $\widehat{\psi}^{n+1}$, cf.~\eqref{eq:GFIit}, and the linearity of the inner product, it holds that
 \begin{align*}
  (\widehat{\psi}^{n+1},\psi_i)_{\Lom}&=(1-\tau^n)(\psi^n,\psi_i)_{\Lom}+\tau^n \gamma^n (\mathsf{G}(\psi^n),\psi_i)_{\Lom}.
 \end{align*} 
 Recall that $\G$ is self-adjoint with respect to the $\Lom$-inner product and $\G(\psi_i)=\nicefrac{\psi_i}{E_i}$, cf.~\eqref{eq:selfadjoint} and~\eqref{eq:Gu}, respectively. This, together with the inductive assumption, implies
 \begin{align*}
   (\widehat{\psi}^{n+1},\psi_i)_{\Lom} = \left(1-\tau^n \left(1-\frac{\gamma^n}{E_i}\right)\right) (\psi^n,\psi_i)_{\Lom}=0,
 \end{align*}
 which proves the claim.
\end{proof}

We emphasize once more that our GFI is (incidentally) the same as the one from~\cite{HenningPeterseim:18}, and thus shares its properties. The following result basically summarizes Lemma~4.7 and Theorem~4.9 in~\cite{HenningPeterseim:18}.

\begin{theorem} \label{thm:timediscreteflow}
Consider the GFI from~\eqref{eq:GFIit} and let $\psi^0 \in \HS$ such that $g(\psi^0)=0$. Then, there exists $\tau_{\max}>0$ such that, for all $\tau^n \leq \tau_{\max}$, 
\[\E(\psi^{n+1}) \leq \E({\psi^n}),\]
i.e.~the energy is dissipative. Moreover, if $0<\tau_{\min} \leq \tau^n \leq \tau_{\max}$ for some $\tau_{\min}>0$ and all $n \geq 0$, then the limit $\E^\star:=\lim_{n \to \infty} \E(\psi^n)$ is well-defined. Furthermore, there exists a subsequence $\{\psi^{n_j}\}_{j \geq 0}$ such that $\psi^{n_j} \to \psi^\star$ strongly in $\H$, where $\psi^\star$ is an eigenfunction of Schr\"{o}dingers equation~\eqref{eq:SEweak} to the eigenvalue 
\[E^\star:=\norm{\G(\psi^\star)}_{\H}^{-1}=\norm{\psi^\star}_{\H}^2=2 \E^\star,\]
i.e. it holds that
 \[\dprod{\E'(\psi^\star),\varphi}=E^\star (\psi^{\star},\varphi)_{\Lom} \qquad \forall \varphi \in \HS.\] 
 The eigenfunction $\psi^\star$ satisfies the equality constraints~\eqref{eq:constraint}, i.e.~
 \[\norm{\psi^\star}_{\Lom}=1 \quad \text{and} \quad (\psi^\star,\psi_i)_{\Lom}=0 \quad \forall i \in \{1,\dots,m\}.\]
The limit of any other convergent subsequence of $\{\psi^n\}$ in $\H$ is likewise an eigenfunction to the eigenvalue $E^\star$, which satisfies all the constraints~\eqref{eq:constraint}.
\end{theorem}

\section{Adaptive gradient flow finite element discretisation}\label{sec:FEM}

We now focus on the adaptive spatial discretisation of the gradient flow iteration~\eqref{eq:GFIit} and the interplay of the two schemes. For that purpose we will employ the strategy from~\cite{HeidStammWihler:19}. In particular, this section is for a large part a replication of~\cite[\S 3]{HeidStammWihler:19} with some minor adaptions.

\subsection{Finite element discretisation}

Consider a sequence of conforming and shape regular partitions $\{\T_N\}_{N\in\mathbb{N}}$ of the domain~$\Omega$ into simplicial elements~$\T_N=\{\kappa\}_{\kappa\in\T_N}$ (i.e.~triangles for $d=2$ and tetrahedra for~$d=3$). Moreover, for a (fixed) polynomial degree~$p\in\mathbb{N}$ and any subset~$\omega \subset\T_N$, we introduce the finite element space 
\[
\V(\omega)=\left\{\varphi \in\HS:\, \varphi|_\kappa\in\poly_p(\kappa), \kappa\in\omega,\,\varphi|_{\Omega\setminus\omega}=0\right\},
\]
with~$\poly_p(\kappa)$ signifying the (local) space of all polynomials of maximal total degree~$p$ on~$\kappa$, $\kappa\in\T_N$. In the sequel, we apply the notation $\X_N:=\V(\mathcal{T}_N)$.

\subsection{Discrete GFI}

Let us define the (space) discrete version of the gradient flow iteration~\eqref{eq:GFIit} on a  finite element subspace $\X_N \subset \HS$. For $\psi \in \HS$ we denote by $\G_N(\psi) \in \X_N$ the unique solution of 
\begin{align*}
 	(\G_N(\psi),\varphi)_{\H}=(\psi,\varphi)_{{\rm L}^2(\Omega)} \qquad \forall \varphi \in \X_N;
\end{align*}
we emphasize that $\G_N:\HS \to \X_N$ is a well-defined linear operator by the Riesz representation theorem, cp.~\eqref{eq:Gdef}. Then, for $\psi_N^0 \in \X_N$ with $g(\psi_N^0)=0$, the space discrete GFI in $\X_N$ is given by
\begin{subequations}\label{eq:discreteGF}
\begin{align} 
 	\psi_N^{n+1}
 	&=\frac{\widehat{\psi}_N^{n+1}}{\norm{\widehat{\psi}_N^{n+1}}_{\Lom}}, \label{eq:discreteGF1}
 	\intertext{where}
 	\widehat{\psi}_N^{n+1}&=(1-\tau_{N}^n) \psi_N^n+\tau_N^n \gamma_N^n\mathsf{G}_N(\psi_N^n),\label{eq:discreteGF2}
\end{align}
\end{subequations}
with $\gamma_N^n:=\frac{1}{(\mathsf{G}_N(\psi_N^n),\psi_N^n)_{\Lom}}$, and a sequence of discrete time steps~$\{\tau_N^n\}_{n \geq 0}$ satisfying $0<\tau_{\min} \leq \tau_N^n \leq \tau_{\max}$ for $n \geq 0$. 

\begin{lemma} \label{lem:discreteconstraint}
Consider a fixed mesh $\T_N$ with corresponding finite element space~$\X_N$. Let $\{\psi_N^n\}_{n \geq 0}\subset\X_N$ be the sequence generated by the discrete GFI~\eqref{eq:discreteGF} with some initial guess $\psi_N^0 \in \HS$ with $g(\psi_N^0)=0$. Then, $g(\psi_N^n)=0$ for all $n \geq 0$.
\end{lemma}

\begin{proof}
The norm constraint trivially holds by~\eqref{eq:discreteGF1}, and thus it remains to prove the orthogonality constraints. Let $P_N:\HS \to \X_N$ denote the orthogonal projection with respect to the $\Lom$-inner product, and define $\psi_{N,i}:=P_N(\psi_i)$, i.e., it holds that
\begin{align} \label{eq:eigenl2}
(\psi_i,\varphi)_{\Lom}=(\psi_{N,i},\varphi)_{\Lom} \qquad \forall \varphi \in \X_N.
\end{align} 
Then, by following the lines in Lemma~\ref{lem:constraint}, we obtain that $g_N(\psi_N^n)=0$ for all $n \geq 0$, where
\begin{align*} 
g_N(\psi):=\left(\norm{\psi}_{\Lom}-1,(\psi_{N,i},\psi)_{\Lom},\dotsc,(\psi_{N,m},\psi)_{\Lom}\right)^t.
\end{align*}
Invoking once more the orthogonality property~\eqref{eq:eigenl2}, we obtain
\begin{align*}
(\psi_N^n,\psi_{N,i})=(\psi_N^n,\psi_i) \qquad \forall i \in \{1,\dotsc,m\}, \, n \geq 0,
\end{align*}
which completes the proof.
\end{proof}



In the discrete space $\X_N$, it holds a similar result as in Theorem~\ref{thm:timediscreteflow}. In particular, the next result is borrowed, and modified in account of the additional constraints, from~\cite[Cor.~4.11]{HenningPeterseim:18}. 

\begin{proposition} \label{prop:discreteconvergence}
Let $\{\tau_N^n\}_{n \geq 0}$ be a time sequence as in Theorem~\ref{thm:timediscreteflow}, and $\{\psi_N^n\}_{n \geq 0}$ be generated by the iteration~\eqref{eq:discreteGF}. Then, the corresponding energies $2 \E(\psi_N^n)$ are strictly monotone decreasing, and the limit $\E_N^\star=\lim_{n \to \infty} \E(\psi_N^n)$ exists. Furthermore, up to subsequences, we have that $\psi_N^{n_j} \to \psi_N^\star$ strongly in $\H$, where $\psi_N^\star \in \X_N$ satisfies $g(\psi_N^\star)=0$ and $\E(\psi_N^\star)=\E_N^\star$. Moreover, $\psi_N^\star$ is a discrete eigenfunction of the corresponding Schr\"{o}dinger equation, i.e.,
 \begin{align} \label{eq:discreteSE}
 \dprod{\E'(\psi_N^\star),\varphi}=E_N^\star (\psi_N^{\star},\varphi)_{\Lom} \qquad \forall \varphi \in \X_N,
 \end{align}
 with $E_N^\star=2 \E_N^\star$. Finally, any other limit point $\Psi_N^\star$ of $\{\psi_N^n\}_{n \geq 0}$ is an eigenstate with $g(\Psi_N^\star)=0$ and energy level $E_N^\star$.
\end{proposition}

For the proof we refer to~\cite{HenningPeterseim:18} and draw attention to Lemma~\ref{lem:discreteconstraint}.

\begin{remark}\label{rem:tauN}
Following the strategy in~\cite[Remark 3.1]{HeidStammWihler:19}, we propose the following time step selection within~\eqref{eq:discreteGF} for practical computations:
\begin{align*} 
 \tau_N^{n} = \max\left\{2^{-m}:\, \E(\psi_N^{n+1}(2^{-m})) < \E(\psi_N^n),\,m\ge 0\right\},\qquad  n \geq 0.
\end{align*}
where, for~$0<s\le 1$, we write $\psi_N^{n+1}(s)$ to denote the output of the discrete GFI \eqref{eq:discreteGF} based on the time step $\tau_N^n=s$ and on the previous approximation~$\psi_N^n$. We emphasize that in view of Proposition~\ref{prop:discreteconvergence} this strategy is well-defined, and guarantees the energy decay in each iterative step. However, for the sake of keeping the computational costs minimal, we fix the time step $\tau=1$ in the \emph{local} GFI in Algorithm~\ref{alg:ref} below, which is in most cases a reasonable choice, cf.~\cite[Remark 3.1]{HeidStammWihler:19}.
\end{remark}

\begin{remark}
Even though Lemma~\ref{lem:discreteconstraint} is in theory true, the orthogonality of $\psi_N^n$ to the given eigenstates $\psi_1,\dotsc,\psi_m$ may get lost when numerical computations are carried out. Thus, in praxis, we slightly modify the discrete GFI~\eqref{eq:discreteGF} by employing the orthogonal projection onto the set 
\[O^{\psi_1^m}_N:=\mathrm{span}\{\psi_{N,1},\dotsc,\psi_{N,m}\}^\perp \subseteq \X_N,\] where $^\perp$ denotes the orthogonal complement with respect to the $\Lom$-inner product and $\psi_{N,i} \in \X_N$ is defined as in the proof of Lemma~\ref{lem:discreteconstraint}. Upon denoting the $\Lom$-orthogonal projection by $\mathrm{P}^{\psi_1^m}_N:\X_N \to O^{\psi_1^m}_N$, the discrete GFI in our computations below is given, for $n \geq 0$, by 
\begin{align} \label{eq:GFIcomp}
\psi_N^{n+1}=\frac{\mathrm{P}^{\psi_1^m}_N(\widehat \psi_N^n)}{\norm{\mathrm{P}^{\psi_1^m}_N(\widehat \psi_N^n)}_{\Lom}},
\end{align}
where $\widehat \psi_N^n$ is defined as in \eqref{eq:discreteGF2}.
\end{remark} 

\subsection{Local energy decay and adaptive mesh refinements}
\label{ssec:LocEn}

As solutions to Schr\"{o}dingers equation may exhibit local features, we shall employ an adaptive mesh refinement procedure. In particular, we will use the strategy from~\cite[\S 3.3]{HeidStammWihler:19}, which, in turn, is based on the one from~\cite{HoustonWihler:16}. In the following, we will recall this adaptive mesh refinement method. Except for some minor modifications that take into account the additional constraints, we more or less copy from~\cite[\S 3.3]{HeidStammWihler:19}. \\

For any $\kappa \in \T_N$, we denote by $\P_\kappa$ the element patch compromising of $\kappa$ and its immediate facewise neighbours. The modified patch $\Pref_\kappa$ is obtained by a red-green refinement of $\P_\kappa$: In particular, $\kappa$ is divided into four triangles by connecting the midpoints of the edges, and any introduced hanging nodes are removed by connecting them to the opposite nodes of the facewise neighbours, see Figure~\ref{fig:ref} for a visualization. We emphasize that this procedure ensures the shape regularity of the mesh, we refer to~\cite{BankShermanWeiser:83}. 

\begin{figure}[t]
\begin{center}
\begin{tabular}{cc}
\includegraphics[scale=0.2]{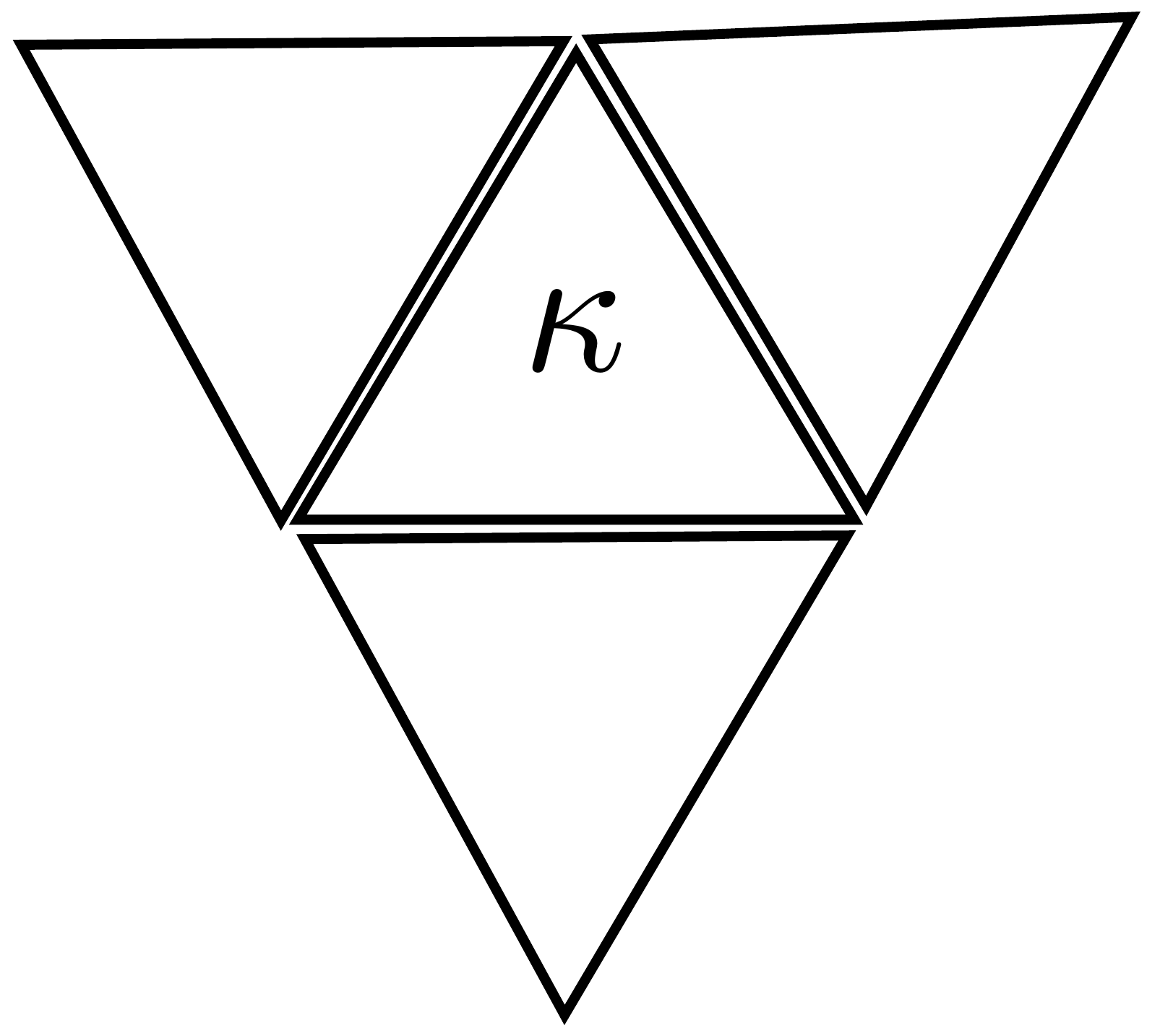} &
\includegraphics[scale=0.2]{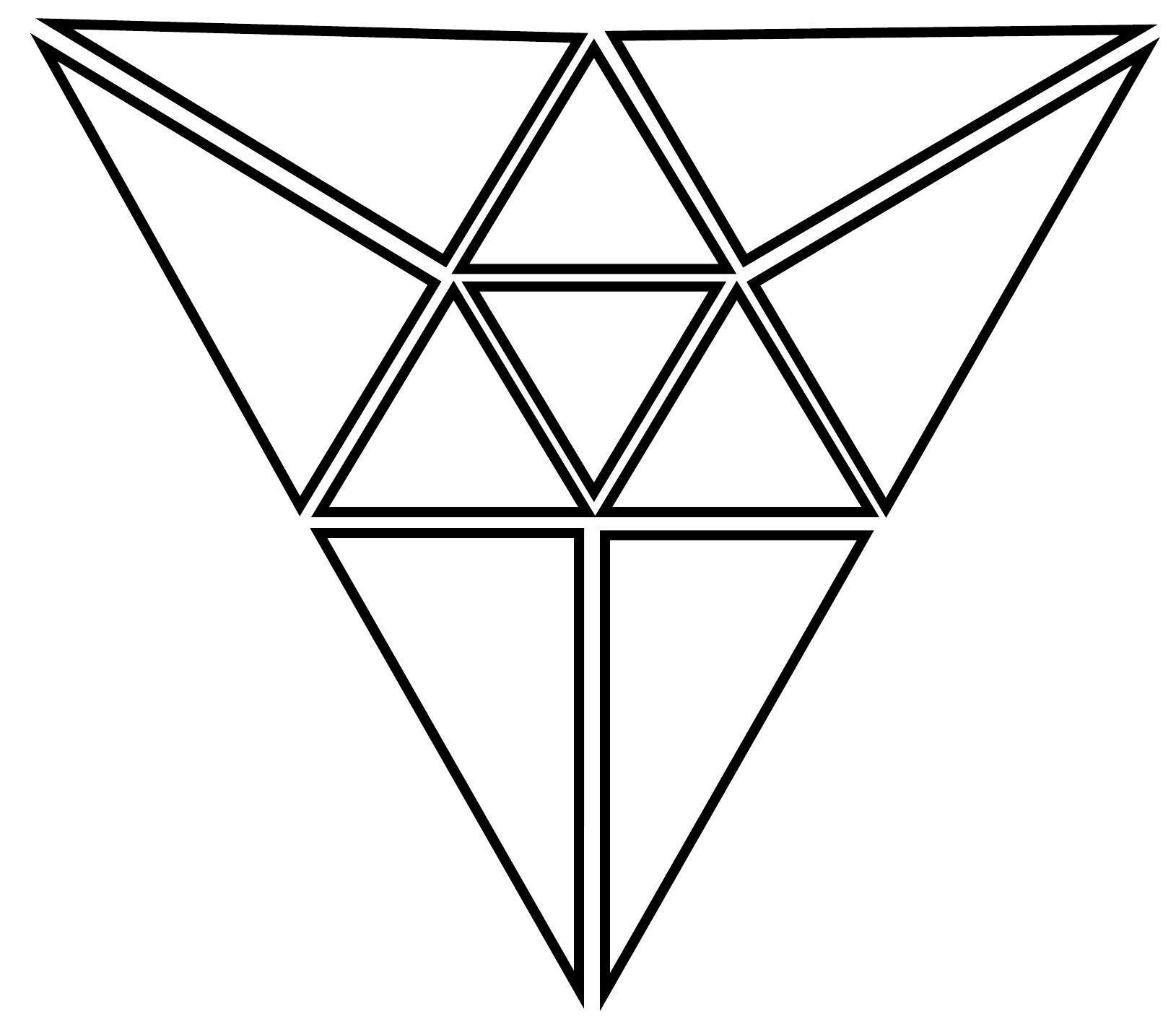}
\end{tabular}
\end{center}
\caption{ Local element patches associated to a triangular element~$\kappa$. Left: Mesh patch~$\P_\kappa$ consisting of the element $\kappa$ and its facewise neighbours. Right: Modified patch $\Pref_\kappa$ constructed based on red-refining $\kappa$ and on green-refining its neighbours. The picture is borrowed from~\cite{HoustonWihler:16}}
\label{fig:ref}
\end{figure}

We further denote by $\{\xi^1_\kappa,\ldots,\xi^{m_\kappa}_\kappa\}$ the basis functions of the locally supported space $\V(\Pref_\kappa)$. Then, for any $\varphi \in \X_N$, we introduce the extended space 
\[
\Vh(\Pref_\kappa;\varphi):=\Span\{\xi^1_\kappa,\ldots,\xi^{m_{\kappa}}_\kappa,\varphi\}.
\] 
Suppose that we have found an accurate approximation $\psi_N^n \in \X_N$ of the discrete Schr\"{o}dinger equation~\eqref{eq:discreteSE}, for some~$n\ge 0$. Then, by performing one \emph{local} discrete GFI with time step $\tau=1$ in $\Vh(\Pref_\kappa;\psi_N^n) \subset \HS$ we obtain a new local approximation, denoted by $\widetilde{\psi}_{N,\kappa}^n \in \Vh(\Pref_\kappa;\psi_N^n)$, with $g(\widetilde{\psi}_{N,\kappa}^n)=0$. We stress that this step, i.e. the local discrete GFI, hardly entails any computational costs since $\Vh(\Pref_\kappa;\psi_N^n) \subset \HS$ is a low dimensional space. Moreover, since the local discrete iterations are independent, they can be computed in parallel for all elements in the mesh. 

By modus operandi of the discrete GFI~\eqref{eq:discreteGF}, the above construction leads in general, see Remark~\ref{rem:tauN}, to the (local) energy decay
\begin{align} \label{eq:localenergydecay}
 -\Delta \E_N^n(\kappa):=\E(\widetilde \psi_{N,\kappa}^n)-\E(\psi_N^n) \leq  0,
\end{align}
for all $\kappa \in \mathcal{T}_N$. The value $\Delta \E_N^n(\kappa)$ indicates the potential energy reduction due to a refinement of the element $\kappa$. This observation motivates the energy-based adaptive mesh refinement procedure outlined in Algorithm~\ref{alg:ref}, which is borrowed from~\cite{HeidStammWihler:19}.

\begin{algorithm}
\caption{Energy-based adaptive mesh refinement~\cite[Alg.~1]{HeidStammWihler:19}}
\label{alg:ref}
\begin{algorithmic}[1]
\State Prescribe a mesh refinement parameter~$\theta\in(0,1)$. 
\State Input a finite element mesh~$\T_N$, and a wavefunction~$\psi^{n}_N \in \X_N$ with $g(\psi_N^n)=0$, for some~$n\ge 1$.
\For {all elements $\kappa\in\T_N$}
	\State\multiline{Perform one discrete GFI-step~\eqref{eq:discreteGF} in the low-dimensional space $\Vh(\Pref_\kappa;\psi_N^n)$ to obtain a potentially improved local approximation~$\widetilde \psi_{N,\kappa}^n$.}  
	\State Compute the local energy decay~$\Delta \E_N^n(\kappa)$ from~\eqref{eq:localenergydecay}.    
\EndFor
\State \textsc{Mark} a subset ~$\mathcal{K} \subset \mathcal{T}_N$ of minimal cardinality which fulfils the D\"orfler marking criterion
\[
\sum_{\kappa \in \mathcal{K}} \Delta \E_N^n(\kappa) \geq \theta \sum_{\kappa \in \mathcal{T}_N} \Delta \E_N^n(\kappa).
\]
\State \textsc{Refine} all elements in~$\mathcal{K}$ for the sake of generating a new mesh~$\T_{N+1}$.
\end{algorithmic}
\end{algorithm}

\subsection{Adaptive strategy}
Next we deal with the decision whether GFI on the given discrete space or local mesh refinements should be given preference. As for the local mesh refinement procedure, cf.~Algorithm~\ref{alg:ref}, we will borrow the strategy from~\cite[\S 3.4]{HeidStammWihler:19}.\\

From a practical viewpoint, once the discrete approximation $\psi_N^n \in \X_N$ is sufficiently close to a discrete eigenfunction of~\eqref{eq:discreteSE}, it is not worth doing more GFI-steps~\eqref{eq:GFIcomp} on the given space, but we should rather refine the mesh $\mathcal{T}_{N}$ to obtain a hierarchically enriched finite element space $\X_{N+1}$. Then, we may embed the final approximation $\psi_N^n \in \X_N$ into the enriched space $\X_{N+1}$ in order to obtain an initial guess $\psi_N^n=:\psi_{N+1}^0 \in \X_{N+1}$. We emphasize that, by Lemma~\ref{lem:discreteconstraint}, $g(\psi_N^n)=0$ for all $n,N \geq 0$ provided that $g(\psi_0^0)=0$.

Similar as in Section~\ref{ssec:LocEn}, the strategy for a a sensible interplay of the discrete GFI~\eqref{eq:GFIcomp} and the local mesh refinement should be solely energy-based. For that purpose, we monitor and compare two energy quantities:
\begin{enumerate}[(a)]
\item The energy increment of each iteration given by
\begin{align*} 
	\incNk := \E(\psi_N^{n-1}) -  \E(\psi_N^n),\qquad n\ge 1;
\end{align*}
\item The energy loss since the latest mesh refinement 
\begin{align*} 
	\dENk := \E(\psi_N^{0}) -  \E(\psi_N^n),\qquad n\ge1.
\end{align*}
\end{enumerate}
We emphasize that, in general, the latest mesh refinement largely contribute to the quantity $\dENk$. Thus, once $\incNk$ is small compared to $\dENk$, we expect a greater benefit from a local mesh refinement than from performing further GFI~\eqref{eq:discreteGF} on the given discrete space $\X_N$. Specifically, for~$n\ge 1$, we stop the GFI on $\X_N$ as soon as
\[
	\incNk \le \gamma\dENk,
\]
for some parameter $0<\gamma<1$.   

This adaptive scheme was proposed in~\cite[Algorithm 2]{HeidStammWihler:19} and will be recalled in Algorithm~\ref{alg:adaptive}. However, we have to slightly adapt this procedure in account of the modified discrete GFI~\eqref{eq:GFIcomp} applied in practical computations. Moreover, we have not included any stopping criterion for the adaptive algorithm, since this may depend on the specific purpose of the application. We, however, refer to~\cite[Algorithm 2]{HeidStammWihler:19} for a possible stopping criterion. 

\begin{algorithm}
\caption{Adaptive finite element gradient flow procedure~\cite[Alg.~2]{HeidStammWihler:19}}
\label{alg:adaptive}
\begin{algorithmic}[1]
\State Prescribe two parameters~$\theta, \gamma\in(0,1)$.
\State Input the already known eigenstates $\psi_1,\dotsc,\psi_m$.
\State Choose a sufficiently fine initial mesh~$\T_0$, and an initial guess $\psi_0^0 \in \X_0$ with $g(\psi^0_0)=0$.
\State Set~$N:=0$.
\Loop 
	\State Set $n:=1$, perform one discrete GFI~\eqref{eq:GFIcomp} in $\X_N$ to obtain $u_{N}^{1}$.
	\State Compute the indicator $\mathrm{inc}_N^1$ (which equals~$\Delta\E_N^1$). 
	\While {$\incNk > \gamma\dENk$} 
		\State Update $n\gets n+1$.
		\State Perform one GFI~\eqref{eq:GFIcomp} in $\X_N$ to obtain $u_{N}^{n}$ (starting from $u_N^{n-1}$).
		\State Compute the indicators $\incNk$ and $\dENk$.
	\EndWhile
		\State \multiline{\textsc{Mark} and adaptively \textsc{refine} the mesh~$\T_N$ using Algorithm~\ref{alg:ref} to generate a new mesh~$\T_{N+1}$.}
		\State Define $\psi_{N+1}^0:=\psi_N^n \in \X_{N+1}$ by canonical embedding $\X_N \hookrightarrow \X_{N+1}$.
		\State Update~$N\gets N+1$.
\EndLoop
\end{algorithmic}
\end{algorithm}

\begin{remark}
The computational work for one loop in Algorithm~\ref{alg:adaptive} scales with $\mathcal O(\textrm{dim}(\mathbb X_N)^\alpha)$, where $\alpha\ge 1$ depends on the global finite element solver, and thus exhibits a similar complexity as standard adaptive finite element discretisation procedures for linear problems (provided that the number of GFI steps remain reasonably modest); we refer~to~\cite[\S 3.5]{HeidStammWihler:19} for a profound analysis of the computational work. However, this requires the parallel computing of the local GFI in Algorithm~\ref{alg:ref}.  
\end{remark}

\section{Numerical Experiments}\label{sec:numerics}
\renewcommand{\thefootnote}{\textit{\alph{footnote}}}

We test our Algorithm~\ref{alg:adaptive} for some experiments in the two dimensional physical space, i.e.,~$d=2$, with Cartesian coordinates denoted by $\bm{x}=(x,y) \in \mathbb{R}^2$. In all examples, we set $\theta=0.5$ and $\gamma=0.1$, and stop the computations once the number of degrees of freedom (i.e.~the dimension of the finite element space~$\X_N$) exceeds~$10^5$. We will only take care of excited states of Schr\"{o}dingers equation, since the ground states of the Experiments~\ref{sec:Lshaped}--\ref{sec:singular} were already approximated in~\cite{HeidStammWihler:19}. In particular, in each of the experiments below, we will consider the ground state approximations from~\cite{HeidStammWihler:19} as the known ground state $\psi_1$. Moreover, the initial mesh $\mathcal{T}_0$ in all our experiments below is coarse and uniform. \\

Since no reference eigenvalues are available in our Experiments~\ref{sec:potwells} and~\ref{sec:singular} below, we will use an a posteriori residual estimator for the evaluation of the computed approximate solution of Schr\"{o}dingers equation~\eqref{eq:SEweak}. Recall that, for an $\Lom$-normalized eigenfunction $\psi$ with corresponding eigenvalue $E$, it holds
\[(\psi,\varphi)_{\H}=E (\psi,\varphi)_{\Lom} \qquad \forall \varphi \in \HS,\]
with $E=\frac{1}{(\G(\psi),\psi)_{\Lom}}$, cf.~\eqref{eq:eigenvalue}. Hence, for any $\psi \in \HS$ with $\norm{\psi}_{\Lom}=1$, we introduce the residual $\R(\psi) \in \mathrm{H}^{-1}(\Omega)$ defined by
\begin{align*} 
\dprod{\R(\psi),\varphi}:=(\psi,\varphi)_{\H}-\frac{1}{(\G(\psi),\psi)_{\Lom}}(\psi,\varphi)_{\Lom}, \qquad \varphi \in \HS;
\end{align*}
we emphasize that this is a neat choice for the residual in the given setting as it naturally incorporates the gradient flow, cf.~\eqref{eq:simplifiedflow}.
By invoking the definition of the two inner products, we obtain
\begin{align*} 
\dprod{\R(\psi),\varphi}=\int_\Omega \nabla \psi \nabla \varphi+(V-e_\psi) \psi \varphi \dx, \qquad \varphi \in \HS,
\end{align*}
where $e_\psi:=(\G(\psi),\psi)_{\Lom}^{-1}$. Then, for any $\psi_N^n \in \X_N$, a standard residual-based a posteriori error analysis (see, e.g.,~\cite{Verfurth:13}) yields the upper bound
\begin{align} \label{eq:residualbound}
\sup_{\varphi \in S_{\H}} \dprod{\R(\psi_N^n),\varphi} \leq C_I \left( \sum_{K \in \mathcal{T}_N} \eta_K^2(\psi_N^n) \right)^{\nicefrac{1}{2}},
\end{align}
where $C_I>0$ is an interpolation constant (only depending on the polynomial degree $p$ and on the shape regularity of the mesh), and 
\begin{align*} 
\eta_K^2(\psi_N^n)=h_K^2 \norm{-\Delta \psi_N^n+(V-e_{\psi_N^n})\psi_N^n}_{\mathrm{L}^2(K)}^2+\frac{1}{2} h_K \norm{\jmp{\nabla \psi_N^n}}_{\mathrm{L^2}(\partial K \setminus \Gamma)}^2, \qquad K \in \mathcal{T}_N
\end{align*}
is a computable bound; here, $h_K:=\mathrm{diam}(K)$ denotes the diameter of $K \in \mathcal{T}_N$. Moreover, for an edge $e \subset \partial K^+ \cap \partial K^{-}$, which is the intersection of two neighbouring elements $K^{\pm} \in \mathcal{T}_N$, we signify by $\jmp{\bm v}|_e=\bm{v}^{+}|_e \cdot \bm{n}_{K^+}+\bm{v}^{-}|_e\cdot \bm{n}_{K^{-}}$ the jump of a (vector-valued) function $\bm{v}$ along~$e$, where $\bm{v}^{\pm}|_e$ denote the traces of the function $\bm{v}$ on the edge $e$ taken from the interior of $K^{\pm}$, respectively, and $\bm{n}_{K^{\pm}}$ are the unit outward normal vectors on $\partial K^{\pm}$, respectively.

\subsection{Laplace EVP on the $\mathrm{L}$-shaped domain} \label{sec:Lshaped}
We begin by testing our algorithm for the Laplace eigenvalue problem, which is to find $\psi \in \HS$ and $E>0$ such that
\begin{align} \label{eq:laplace}
\int_\Omega \nabla \psi \nabla \varphi \dx= E (\psi,\varphi)_{\Lom} \qquad \forall \varphi \in \HS;
 \end{align}
here, $\Omega:=(0,2)^2 \setminus [1,2] \times [0,1]$ is an $\mathrm{L}$-shaped domain. 
This is a well examined problem, for whose eigenvalues (sharp) lower and upper bounds are known. In particular, for the evaluation of the eigenvalues approximation in our computations, we will use the lower bounds from~\cite{FoxMoler:67} as the reference values. 
\begin{enumerate}[(a)]
\item We start with the ground state $\psi_1$ from~\cite{HeidStammWihler:19}. In order to find a new eigenfunction of~\eqref{eq:laplace}, we run Algorithm~\ref{alg:adaptive} with the initial guess $\psi_0^0:=\mathrm{P}_0^{\psi_1^1}(\widetilde \psi_0^0) \norm{\mathrm{P}_0^{\psi_1^1}(\widetilde \psi_0^0)}_{\Lom}^{-1}$, where $\widetilde \psi_0^0 \in \X_0$ is the linear interpolant of the function $(x,y) \mapsto |\mathrm{sin}(\pi x) \mathrm{sin}(\pi y)| \mathrm{sign}(y-1)$ in the element nodes. In Figure~\ref{fig:Lshaped2} (left), we observe the optimal convergence rate for the eigenvalue approximation. The approximated eigenstate $\psi_2$ is visualized in Figure~\ref{fig:Lshaped2} (right).

\begin{figure}
	\hfill
	\includegraphics[width=0.49\textwidth]{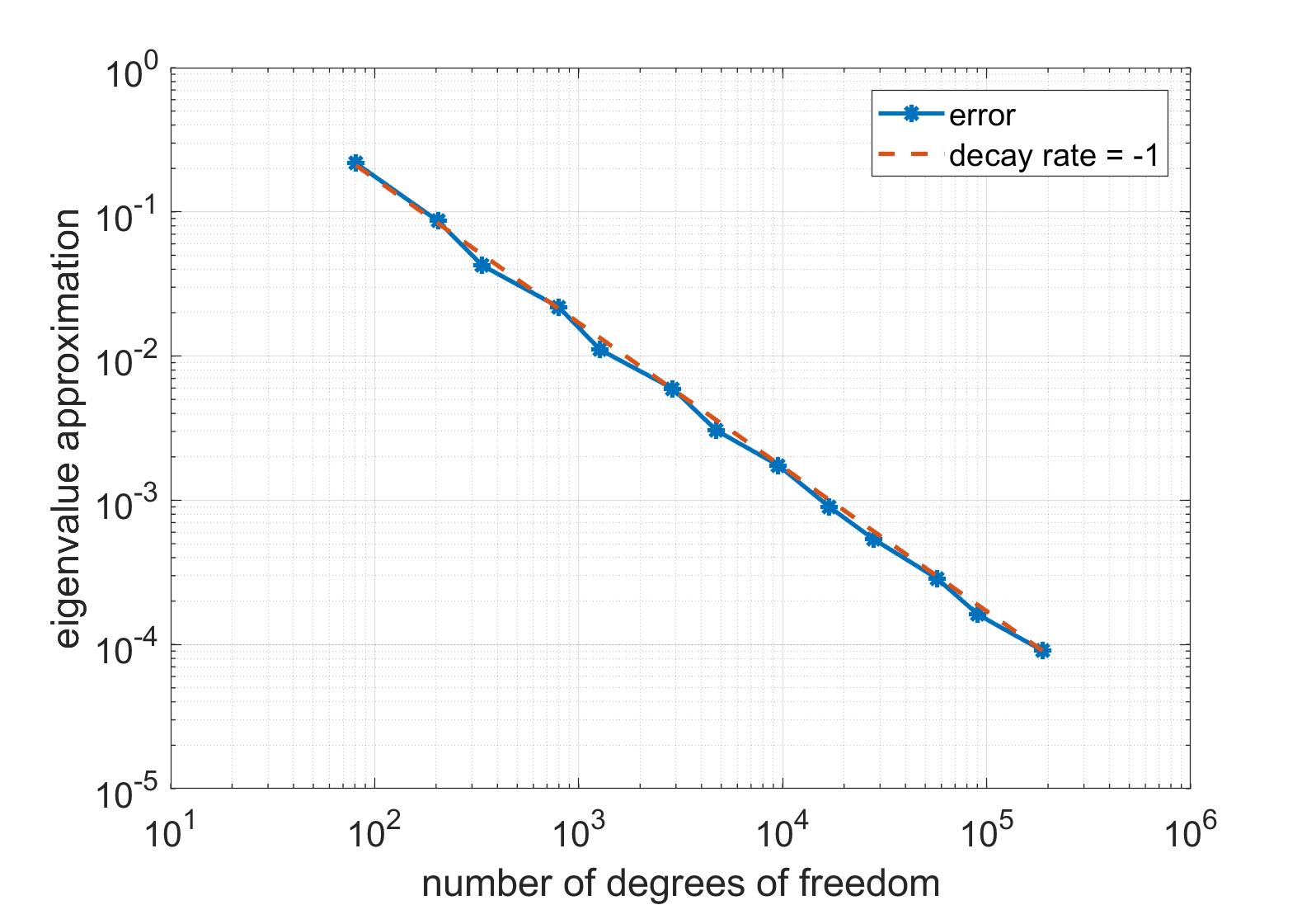}
	\hfill
	\includegraphics[width=0.49\textwidth]{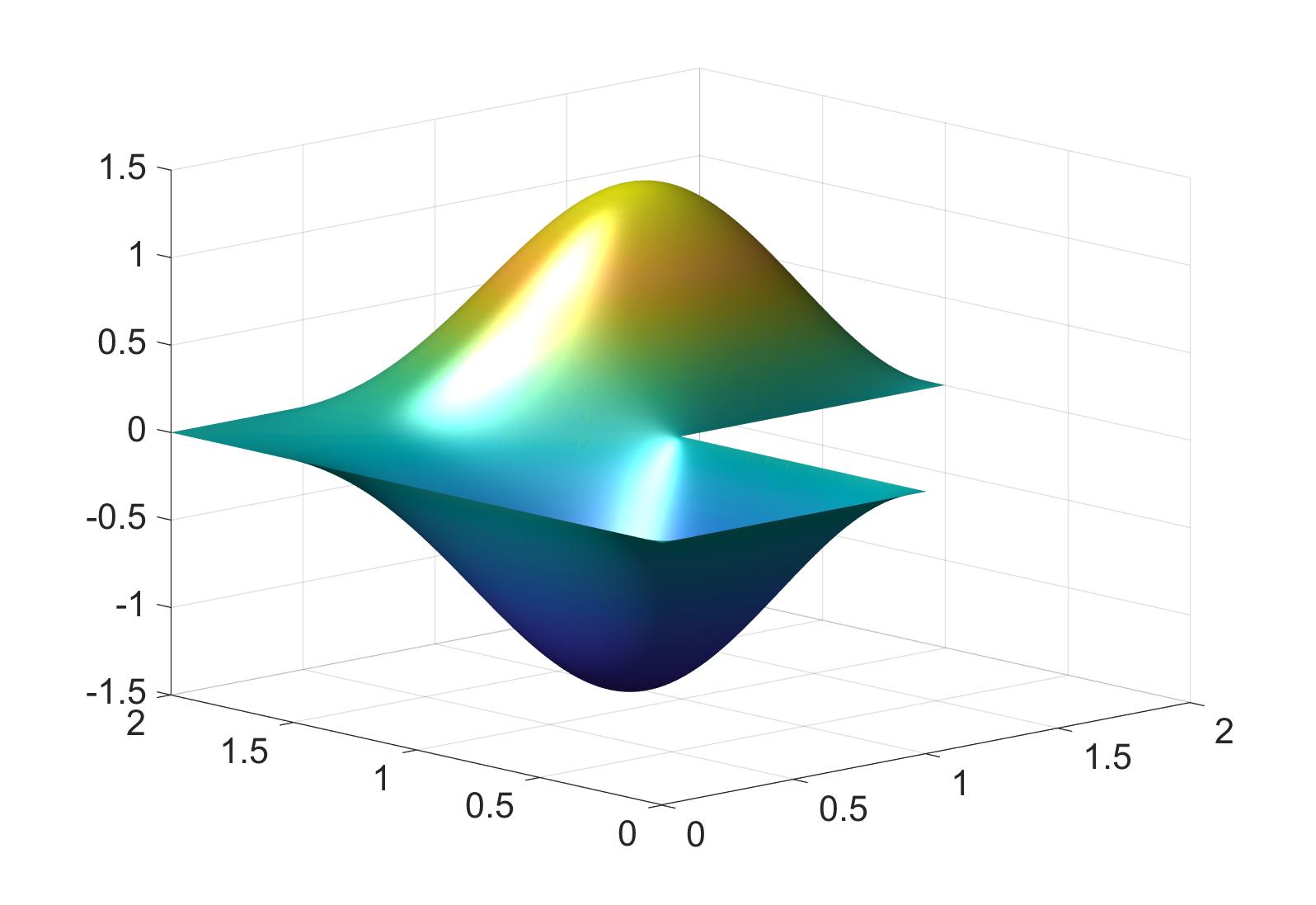}
	\hfill
	\caption{Experiment~\ref{sec:Lshaped} (a). Left: Convergence plot for the eigenvalue approximation. Right: Approximated eigenstate $\psi_2$.}
	\label{fig:Lshaped2}
\end{figure}

\item Next, we set $\psi_0^0:=\mathrm{P}_0^{\psi_1^2}(\widetilde \psi_0^0) \norm{\mathrm{P}_0^{\psi_1^2}(\widetilde \psi_0^0)}_{\Lom}^{-1}$, where $\widetilde \psi_0^0 \in \X_0$ denotes the linear interpolant of the function $(x,y) \mapsto \mathrm{sin}(\pi x) \mathrm{sin}(\pi y)$ in the element nodes. As before, the error of the eigenvalue decays at an (almost) optimal rate of $\mathcal{O}(|\mathcal{T}_N|^{-1})$, see Figure~\ref{fig:Lshaped3} (left). Figure~\ref{fig:Lshaped3} depicts the corresponding eigenstate approximation $\psi_3$.

\begin{figure}
	\hfill
	\includegraphics[width=0.49\textwidth]{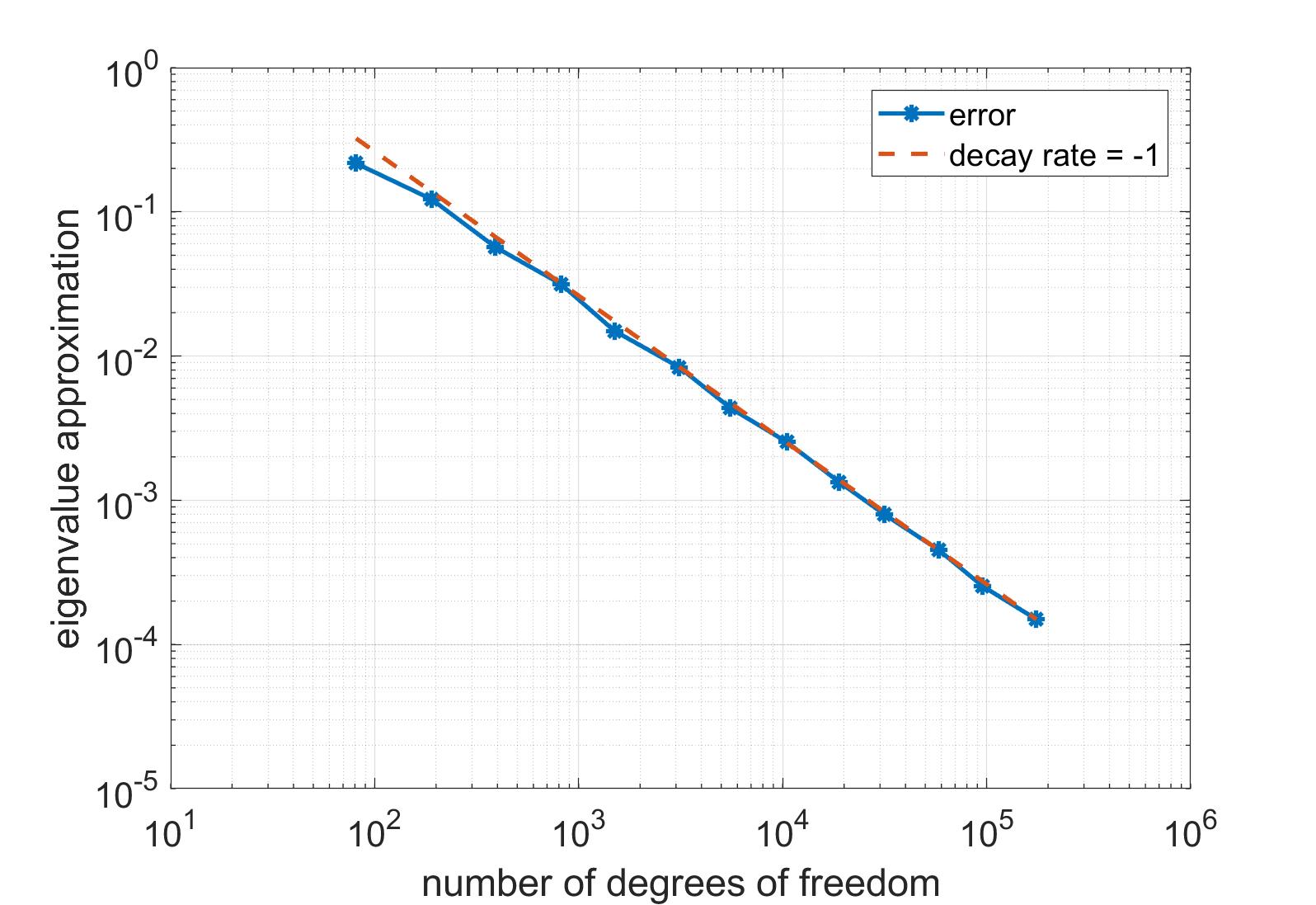}
	\hfill
	\includegraphics[width=0.49\textwidth]{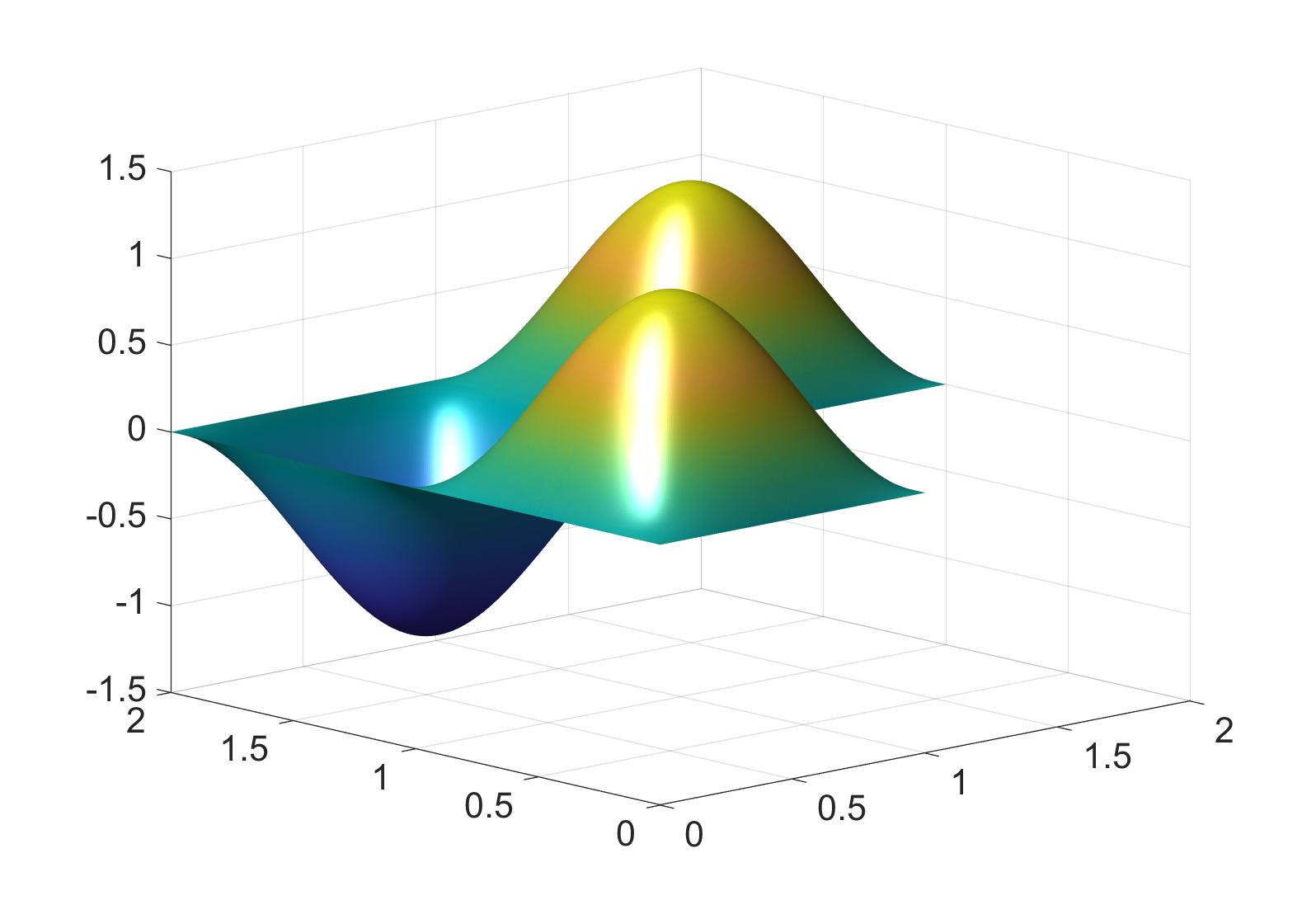}
	\hfill
	\caption{Experiment~\ref{sec:Lshaped} (b). Left: Convergence plot for the eigenvalue approximation. Right: Approximated eigenstate $\psi_3$.}
	\label{fig:Lshaped3}
\end{figure}

\item Finally, we choose the initial guess $\psi_0^0=\mathrm{P}_0^{\psi_1^3}(\widetilde \psi_0^0) \norm{\mathrm{P}_0^{\psi_1^3}(\widetilde \psi_0^0)}_{\Lom}^{-1}$, where $\widetilde{\psi}_0^0 \in \X_0$ is such that $\widetilde \psi_0^0(\bm{x})=c$ for any node $\bm{x}$ in the interior of the corresponding initial mesh $\mathcal{T}_0$ for some constant $c >0$; we remark that this was also  the initial guess, up to the projection, for the ground state computation in~\cite[Exp.~4.1]{HeidStammWihler:19}. Once more, the adaptive algorithm exhibits optimal convergence rate for the eigenvalue approximation, see Figure~\ref{fig:Lshaped4} (left). The approximated eigenstate is plotted in Figure~\ref{fig:Lshaped4} (right).

\begin{figure}
	\hfill
	\includegraphics[width=0.49\textwidth]{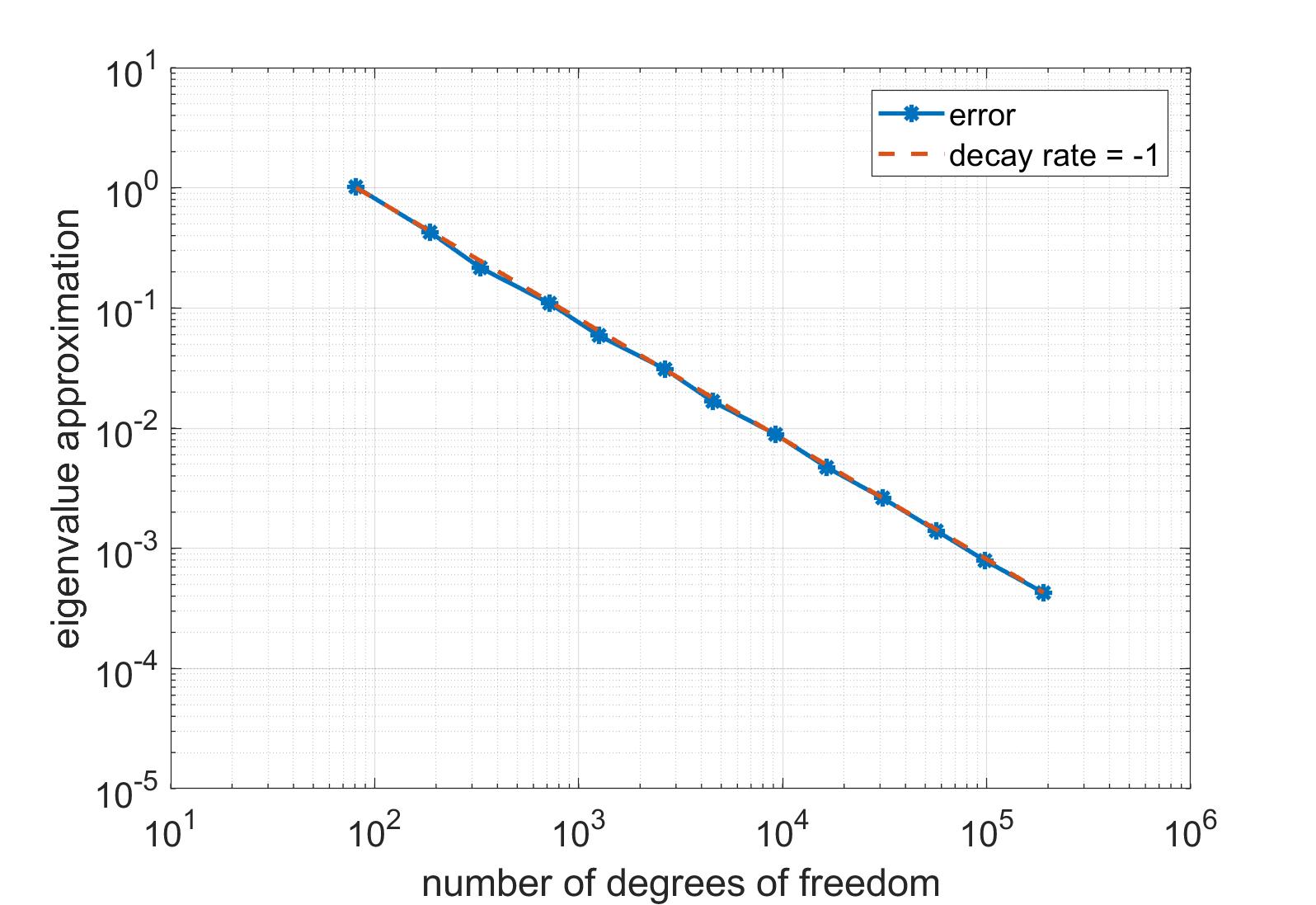}
	\hfill
	\includegraphics[width=0.49\textwidth]{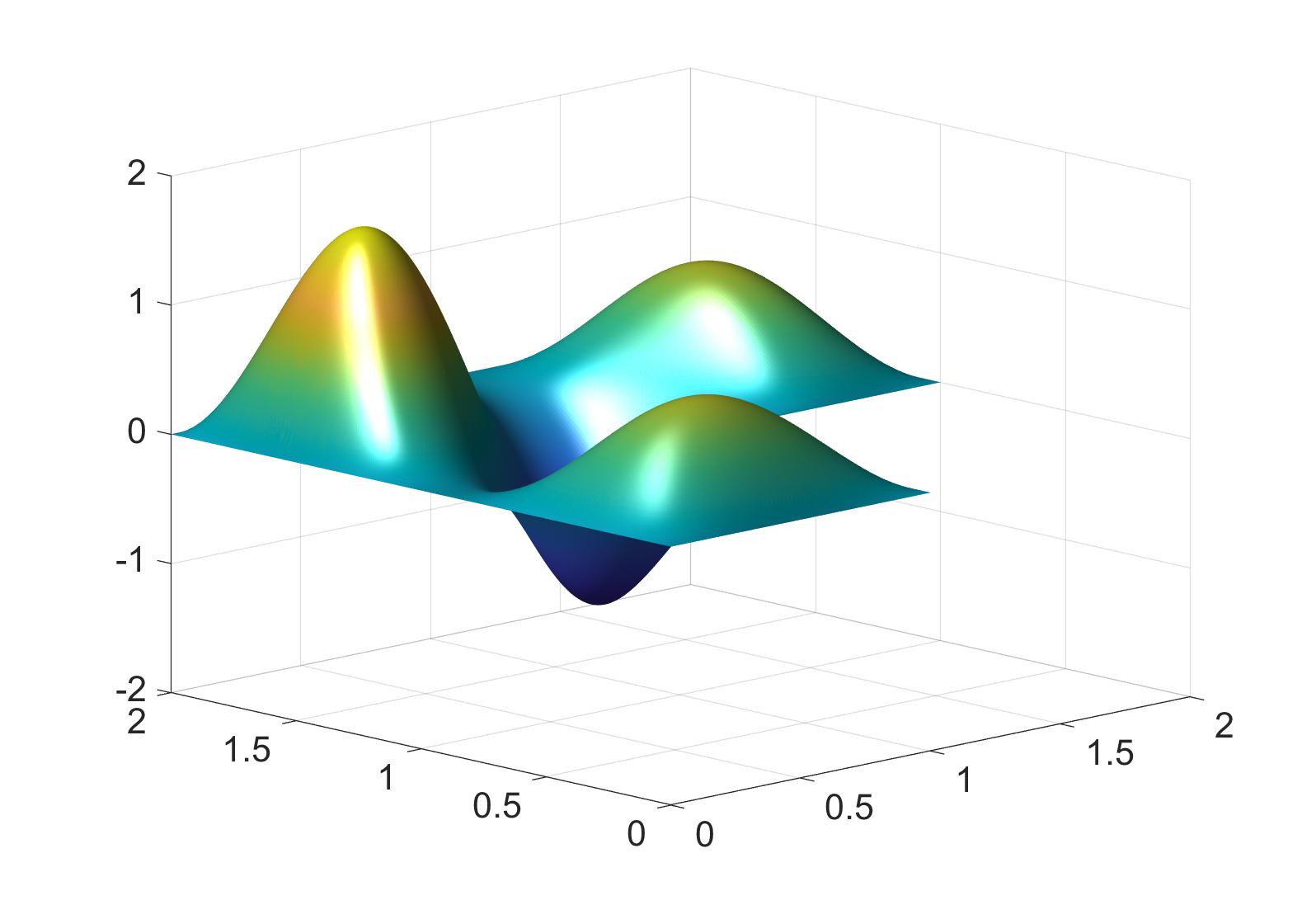}
	\hfill
	\caption{Experiment~\ref{sec:Lshaped} (c). Left: Convergence plot for the eigenvalue approximation. Right: Approximated eigenstate $\psi_4$.}
	\label{fig:Lshaped4}
\end{figure}

\end{enumerate}

\subsection{Schr\"{o}dingers equation with potential wells} \label{sec:potwells}

We consider the Experiment~4.4 from~\cite{HeidStammWihler:19}, where the potential $V$ is given by the sum of four Gaussian bells, see Figure~\ref{fig:potwells}; this experiment was originally examined in~\cite[Exp.~4.2]{LinStamm:17}, however, we employed a (constant) shift in the potential such that $V \geq 0$ in the underlying domain~$\Omega:=(0,2\pi)^2$. In this application, the initial guess will be $\psi_0^0=\mathrm{P}_0^{\psi_1^m}(\widetilde \psi_0^0) \norm{\mathrm{P}_0^{\psi_1^m}(\widetilde \psi_0^0)}_{\Lom}^{-1}$, with $m=1$ in (a) and $m=2$ in (b), where $\widetilde{\psi}_0^0 \in \X_0$ is such that $\widetilde \psi_0^0(\bm{x})=c$ for any node $\bm{x}$ in the interior of the corresponding initial mesh $\mathcal{T}_0$ for some constant $c >0$; moreover, $\psi_1$ is the ground state borrowed from~\cite[Exp.~4.4]{HeidStammWihler:19} and $\psi_2$ the excited state approximated in (a). As no reference values for the energies of the eigenstates are available, we will plot the residual bound~\eqref{eq:residualbound}, with ad-hoc selection $C_I=1$, against the number of degrees of freedom.

\begin{figure}
	\includegraphics[width=0.49\textwidth]{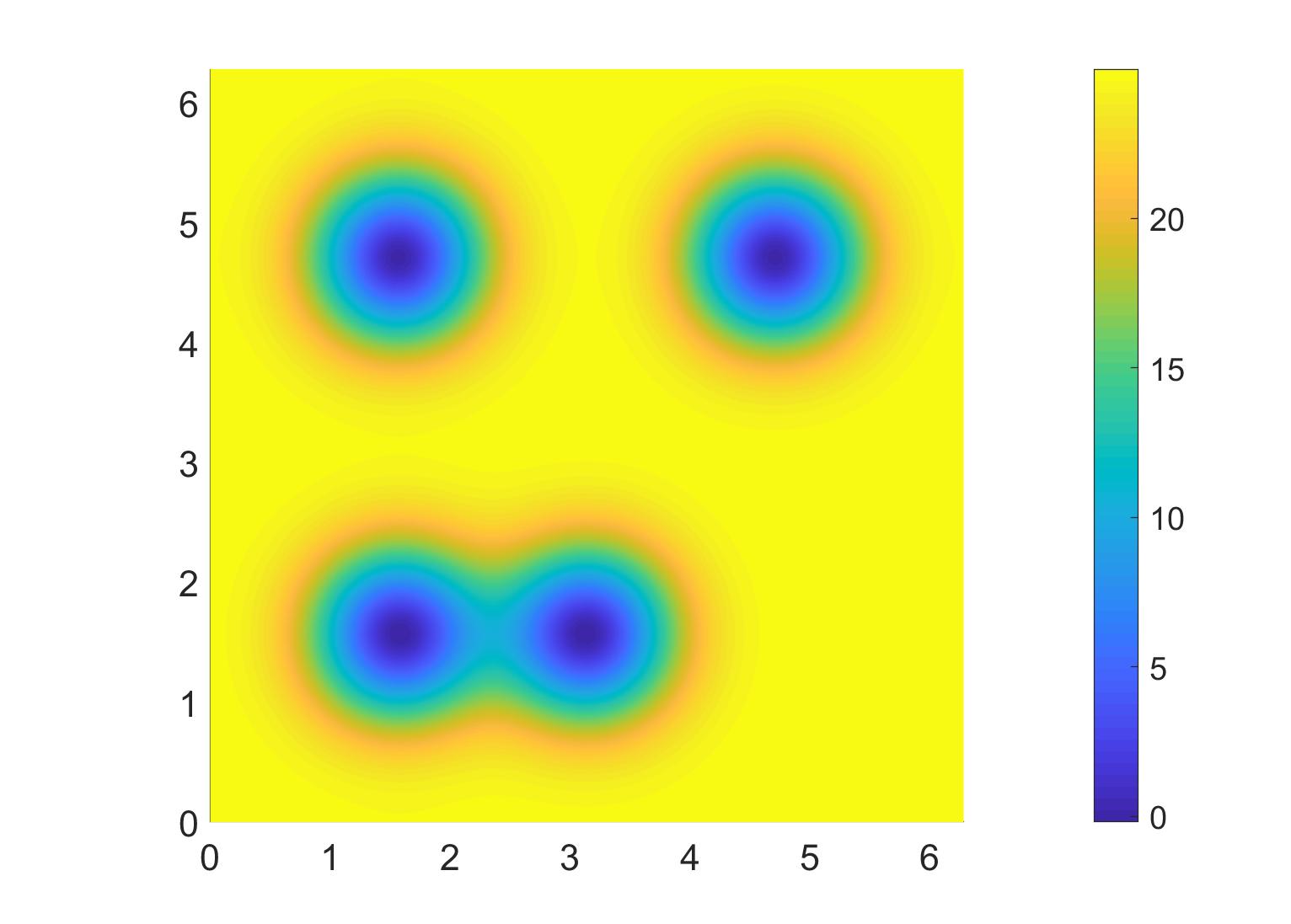}
	\caption{Experiment~\ref{sec:potwells}. Potential function~$V$ consisting of four Gaussian bells.}
	\label{fig:potwells}
\end{figure}

\begin{enumerate}[(a)]
\item In Figure~\ref{fig:potwellsconvergence} (left), we see that the residual decays at an optimal rate of $\mathcal{O}(|\mathcal{T}_N|^{-\nicefrac{1}{2}})$. The probability density of the approximated eigenstate $\psi_2$, which is given by $|\psi_2|$, is concentrated in the two wells in the upper part of the domain, see Figure~\ref{fig:potwells_psi2} (left). This was properly resolved by the adaptive mesh refinements, see Figure~\ref{fig:potwells_psi2} (right).
\item The algorithm exhibits close to optimal convergence rate for the residual, see Figure~\ref{fig:potwellsconvergence} (right). In contrast to before, there is also a small amplitude of $\psi_3$ in each of the two connected wells, see Figure~\ref{fig:potwells_psi3} (left). This was well recognised by the adaptive algorithm as illustrated in Figure~\ref{fig:potwells_psi3} (right).
\end{enumerate}

\begin{figure}
	\hfill
	\includegraphics[width=0.49\textwidth]{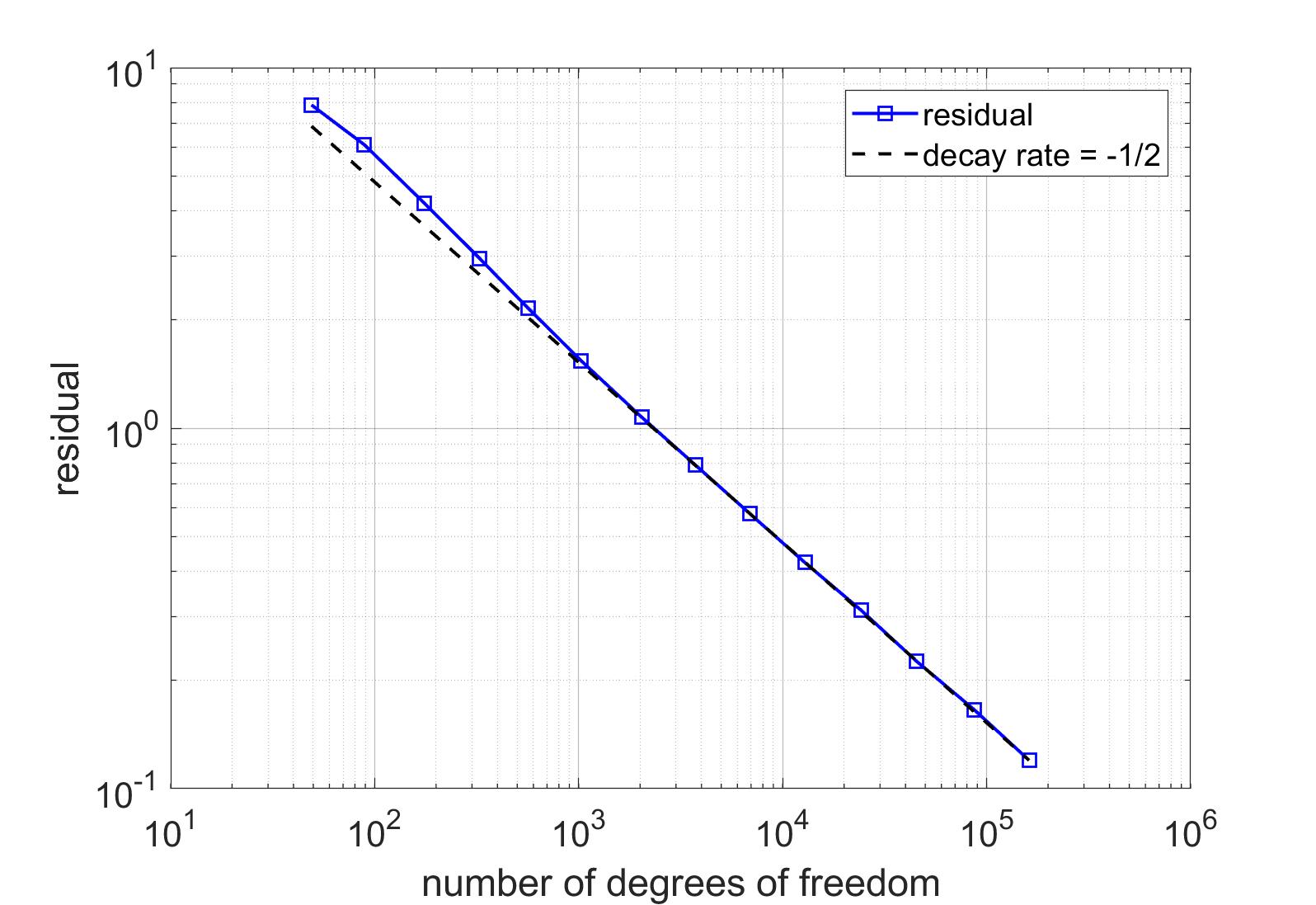}
	\hfill
	\includegraphics[width=0.49\textwidth]{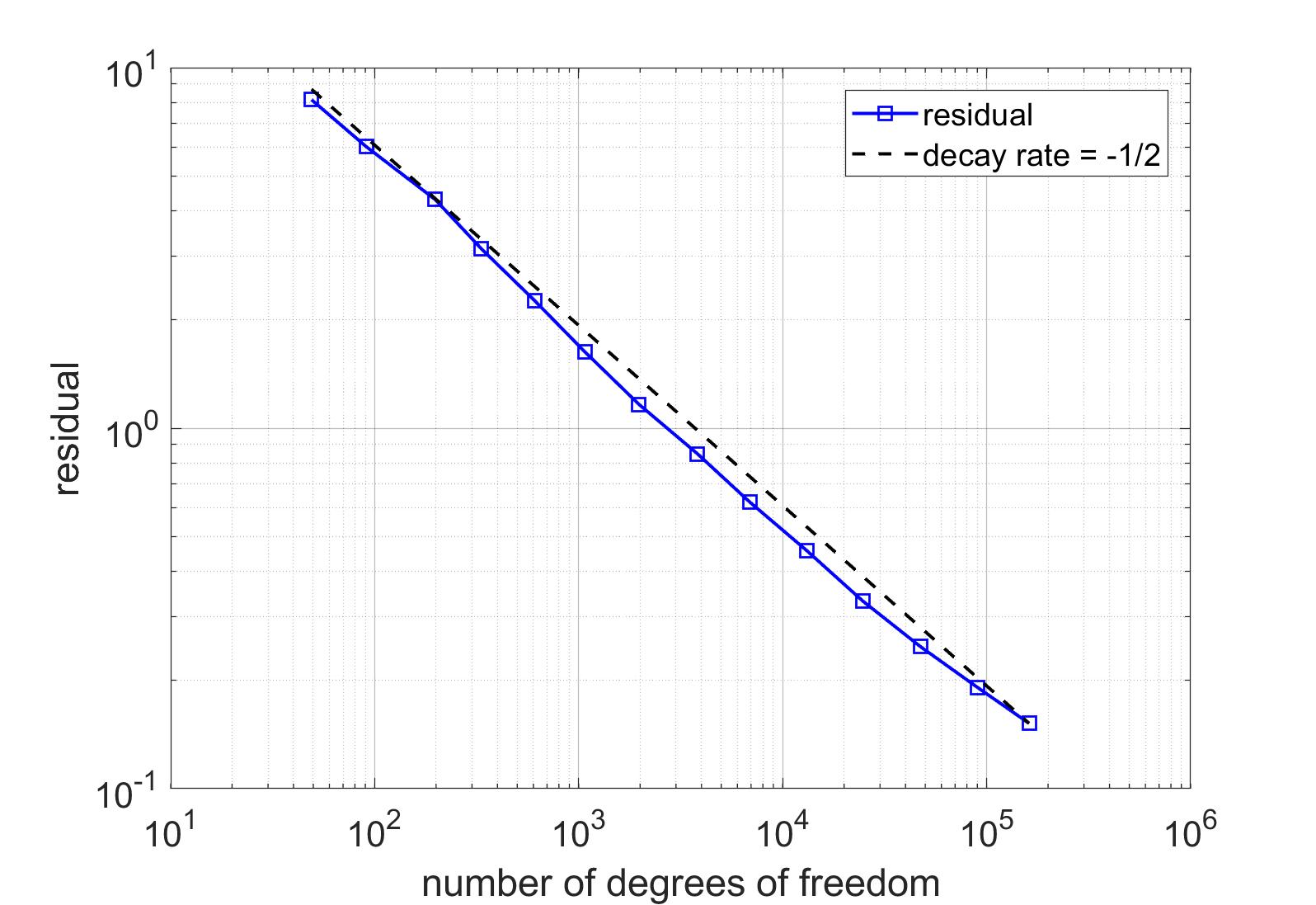}
	\hfill
	\caption{Experiment~\ref{sec:potwells}. Residual decay for the eigenstate approximation of $\psi_2$ (left) and $\psi_3$ (right), respectively.}
	\label{fig:potwellsconvergence}
\end{figure}

\begin{figure}
	\hfill
	\includegraphics[width=0.49\textwidth]{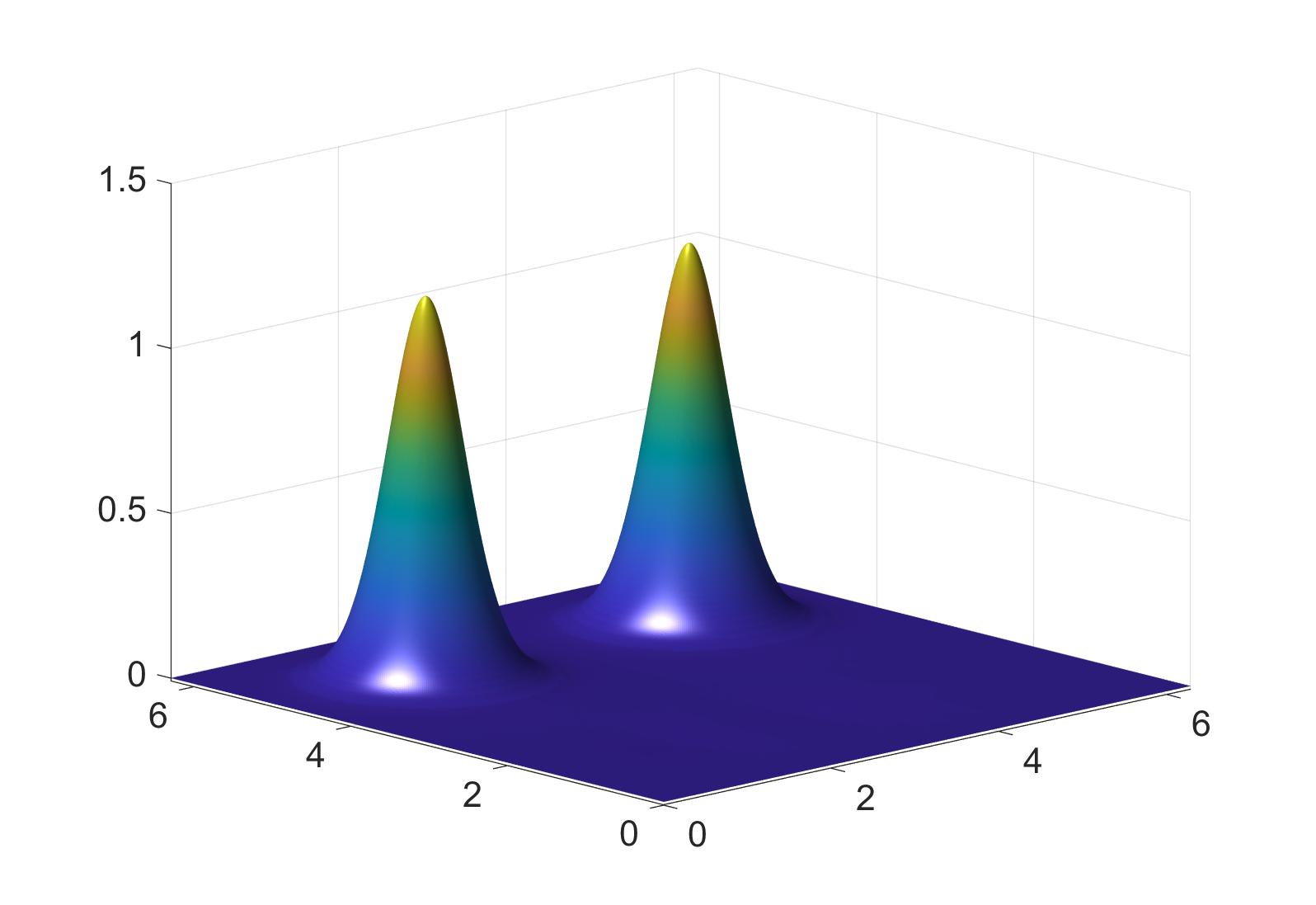}
	\hfill
	\includegraphics[width=0.49\textwidth]{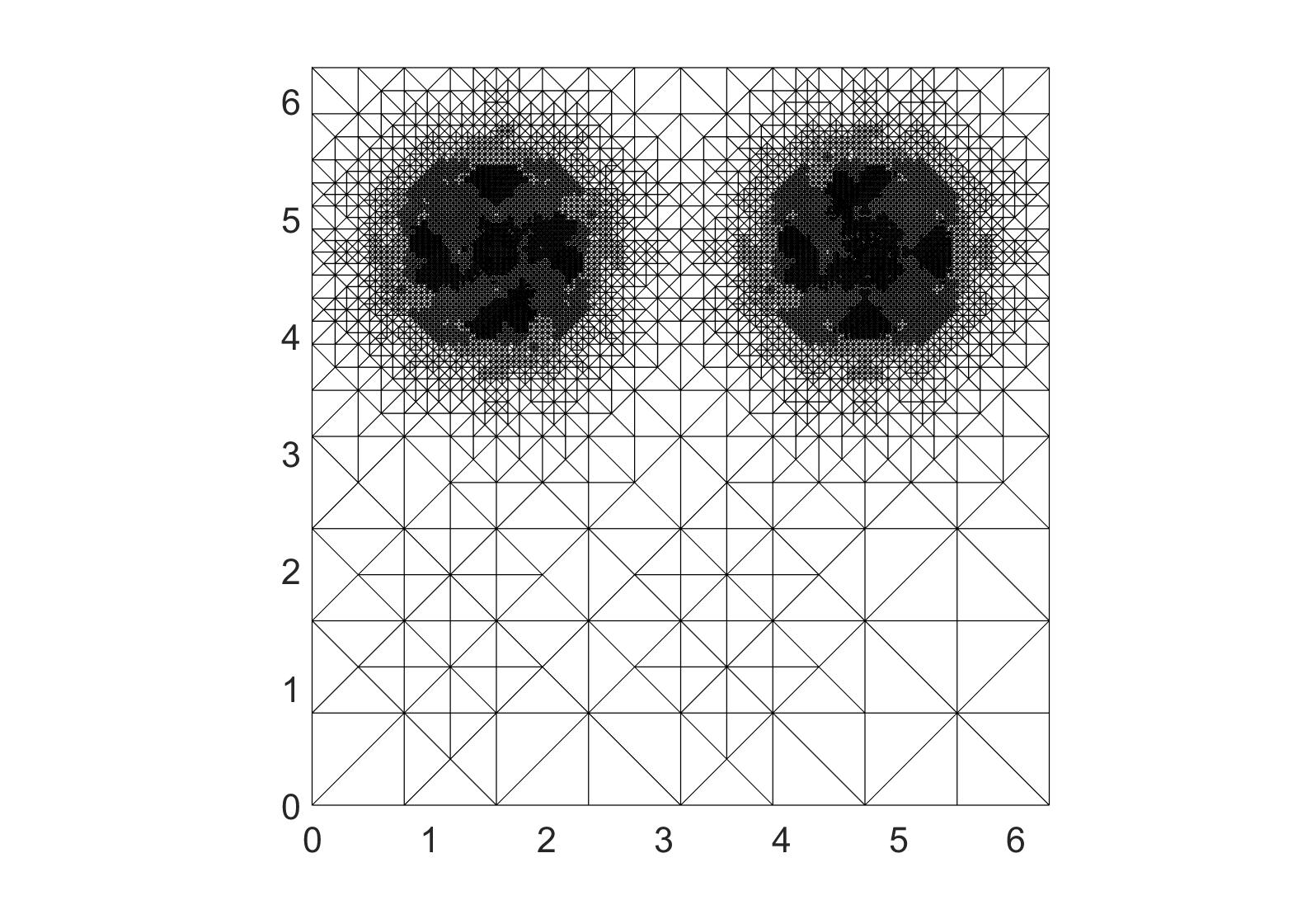}
	\hfill
	\caption{Experiment~\ref{sec:potwells} (a). Left: Approximated eigenstate $\psi_2$. Right: Adaptively refined mesh after 9 refinements.}
	\label{fig:potwells_psi2}
\end{figure}

\begin{figure}
	\hfill
	\includegraphics[width=0.49\textwidth]{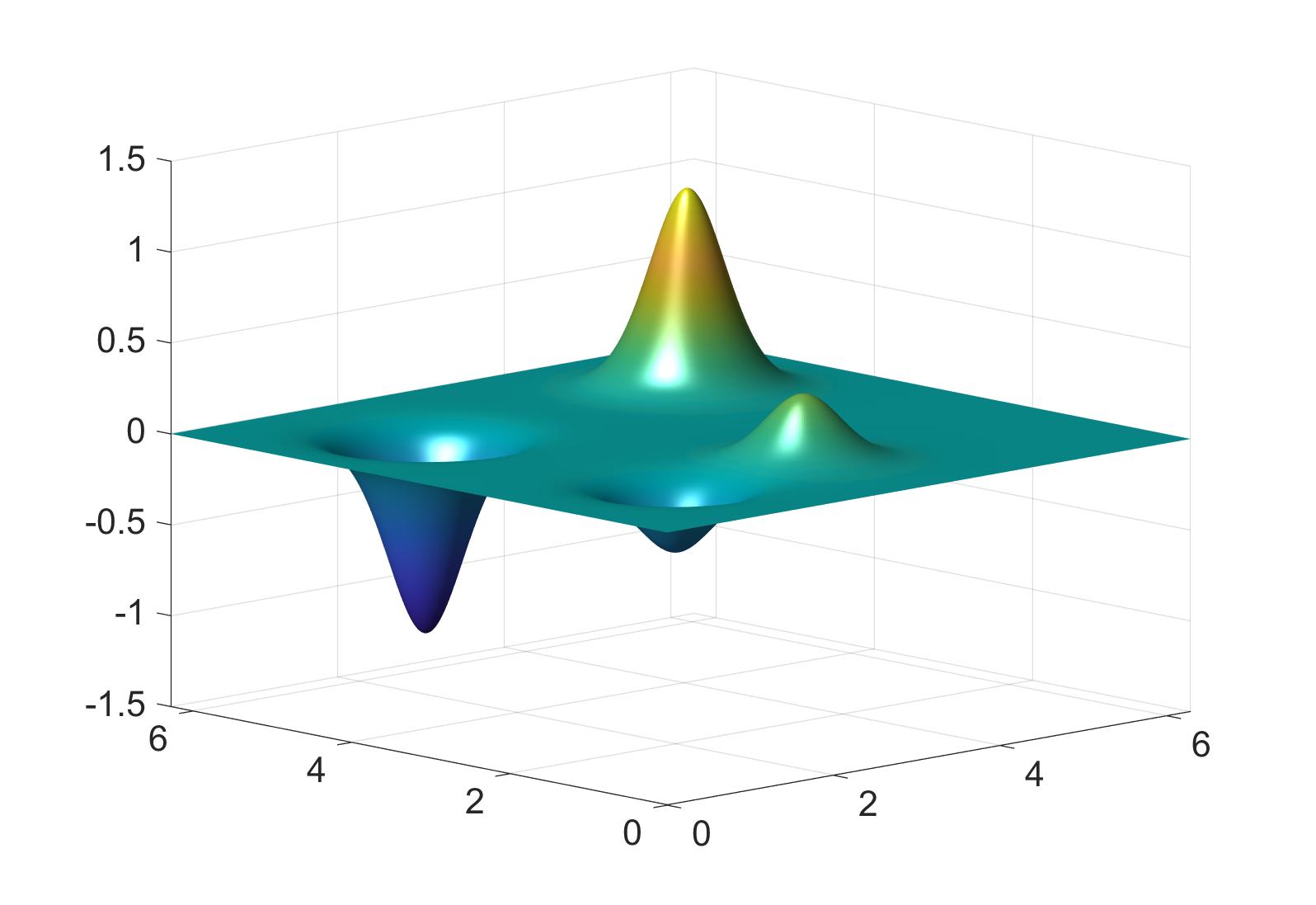}
	\hfill
	\includegraphics[width=0.49\textwidth]{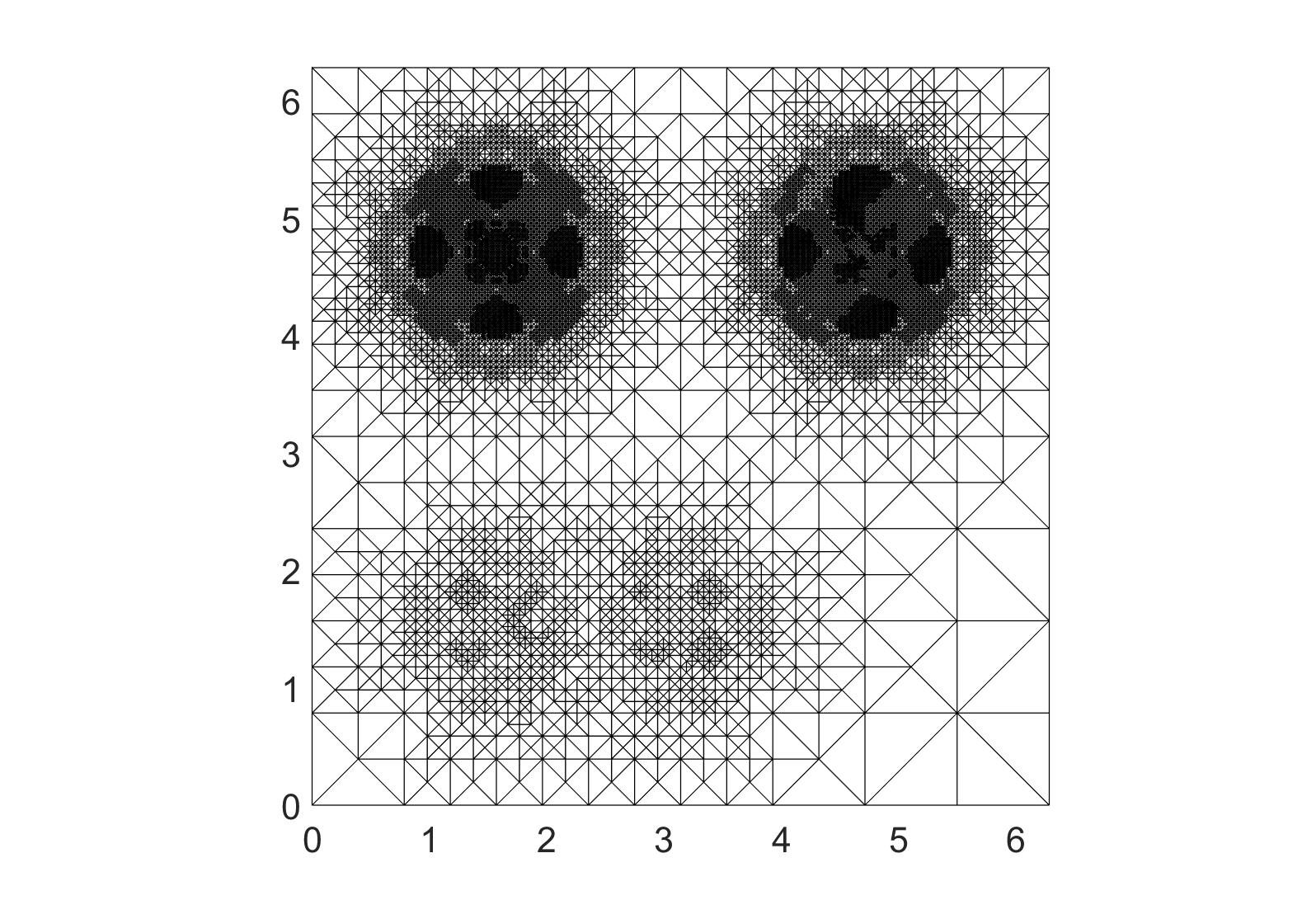}
	\hfill
	\caption{Experiment~\ref{sec:potwells} (b). Left: Approximated eigenstate $\psi_3$. Right: Adaptively refined mesh after 9 refinements.}
	\label{fig:potwells_psi3}
\end{figure}

\subsection{Schr\"{o}dingers equation with singular potential $V$.} \label{sec:singular}

In our last example, we consider the potential $V(\bm{x})=(2|\bm{x}|)^{-1}$ which features a severe point singularity at the origin~$\bm{0}=(0,0)$, cp.~\cite[Exp.~4.3]{HeidStammWihler:19}. Then, the corresponding functional is given by
\begin{align*}
 \E(u)=\frac14\int_\Omega \left( |\nabla u|^2+|\bm{x}|^{-1}|u|^2\right) \dx,
\end{align*}
with $\Omega := \left(\nicefrac{-1}{2},\nicefrac{1}{2}\right)^2$. As in the experiment before, we always choose the initial state $\psi_0^0:=\mathrm{P}_0^{\psi_1^m}(\widetilde \psi_0^0) \norm{\mathrm{P}_0^{\psi_1^m}(\widetilde \psi_0^0)}_{\Lom}^{-1}$, where $\widetilde{\psi}_0^0 \in \X_0$ is such that $\widetilde \psi_0^0(\bm{x})=c$ for any node $\bm{x}$ in the interior of the corresponding initial mesh $\mathcal{T}_0$ for some constant $c >0$. We run the Algorithm~\ref{alg:adaptive} three times in succession in order to approximate, based on one another, eigenstates $\psi_2,\psi_3$ and $\psi_4$. Figures~\ref{fig:Singular2}--\ref{fig:Singular4} (left) indicate that the decay rate of the residual bound~\eqref{eq:residualbound} in the Experiments~\ref{sec:singular} (a)--(c) is once more (close to) optimal. The corresponding eigenstates are displayed in Figures~\ref{fig:Singular2}--\ref{fig:Singular4} (right). It seems that the approximated eigenstates $\psi_2$ and $\psi_3$ coincide up to rotation, compare Figure~\ref{fig:Singular2} (right) and~\ref{fig:Singular3} (right).
 
\begin{figure}
	\hfill
	\includegraphics[width=0.49\textwidth]{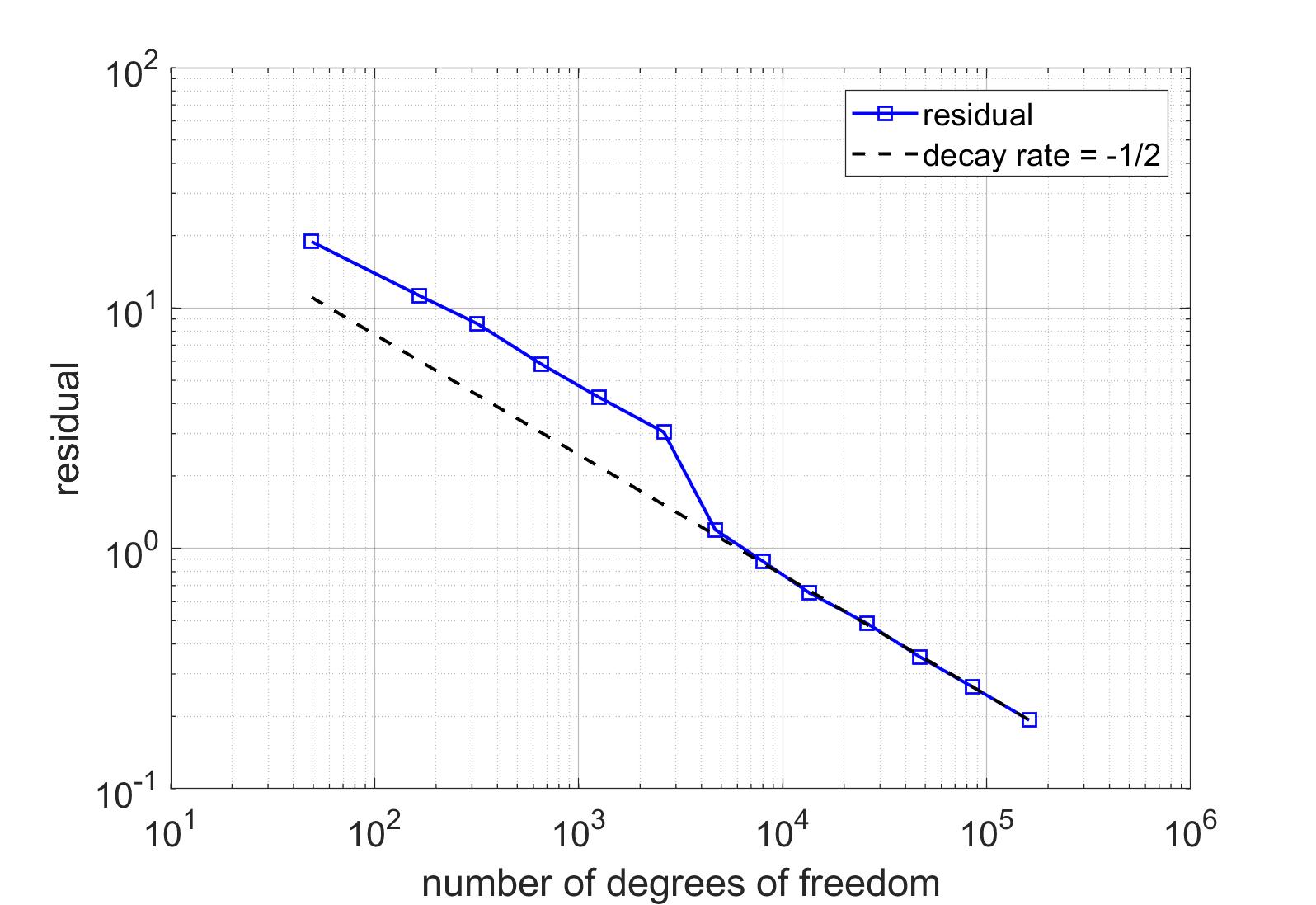}
	\hfill
	\includegraphics[width=0.49\textwidth]{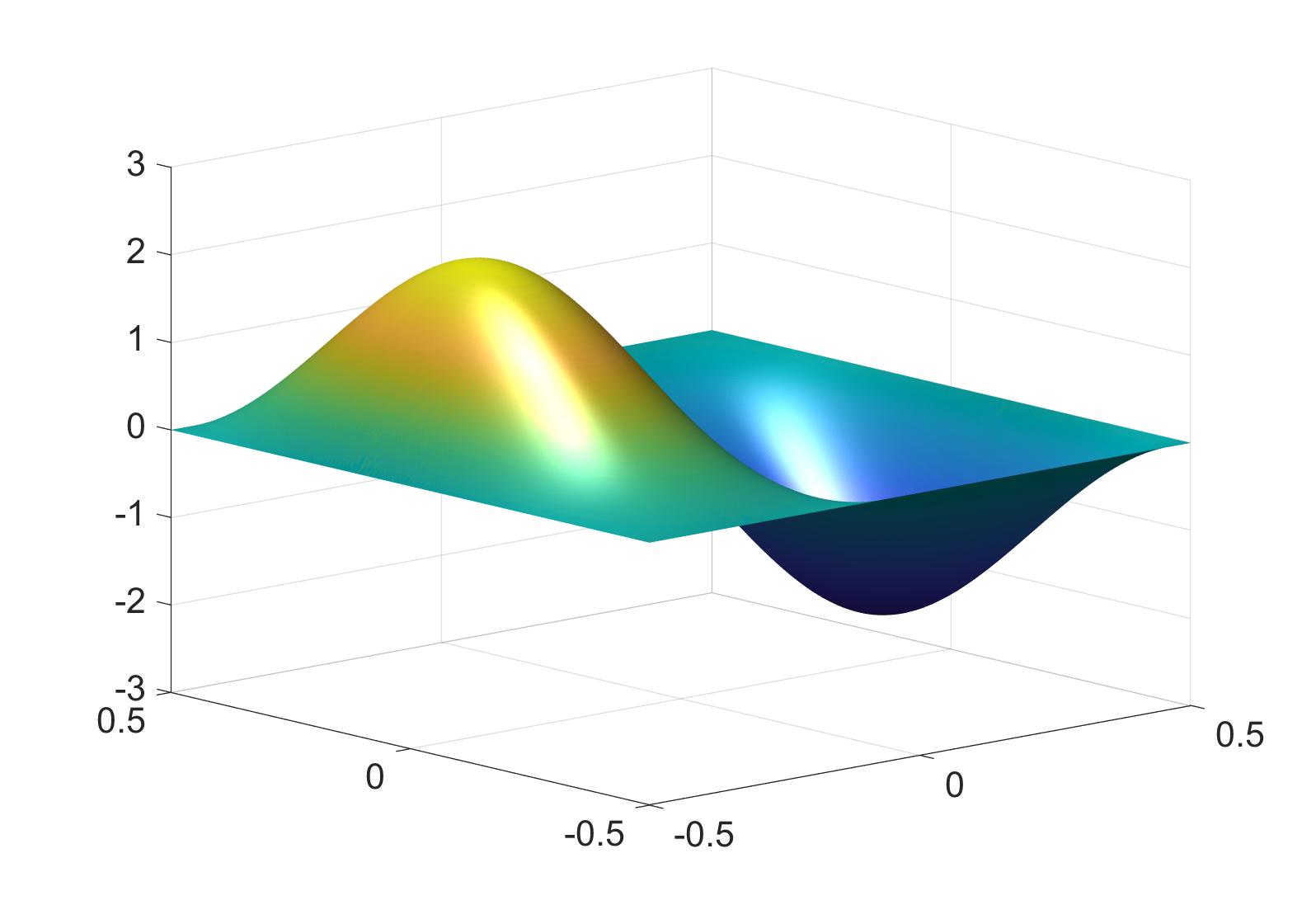}
	\hfill
	\caption{Experiment~\ref{sec:singular} (a). Left: Residual decay for the eigenstate approximation of $\psi_2$. Right: Approximated eigenstate $\psi_2$.}
	\label{fig:Singular2}
\end{figure}

\begin{figure}
	\hfill
	\includegraphics[width=0.49\textwidth]{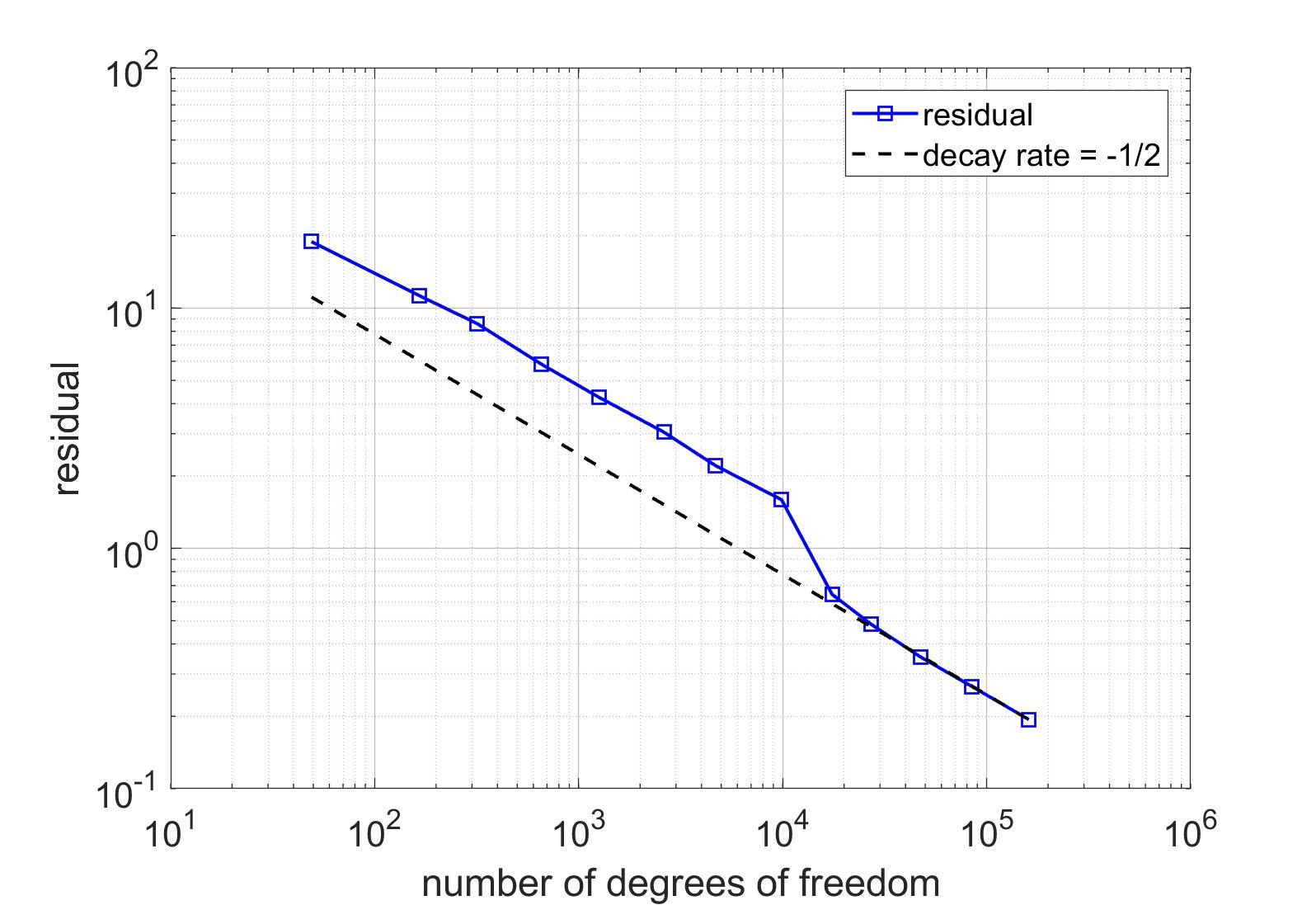}
	\hfill
	\includegraphics[width=0.49\textwidth]{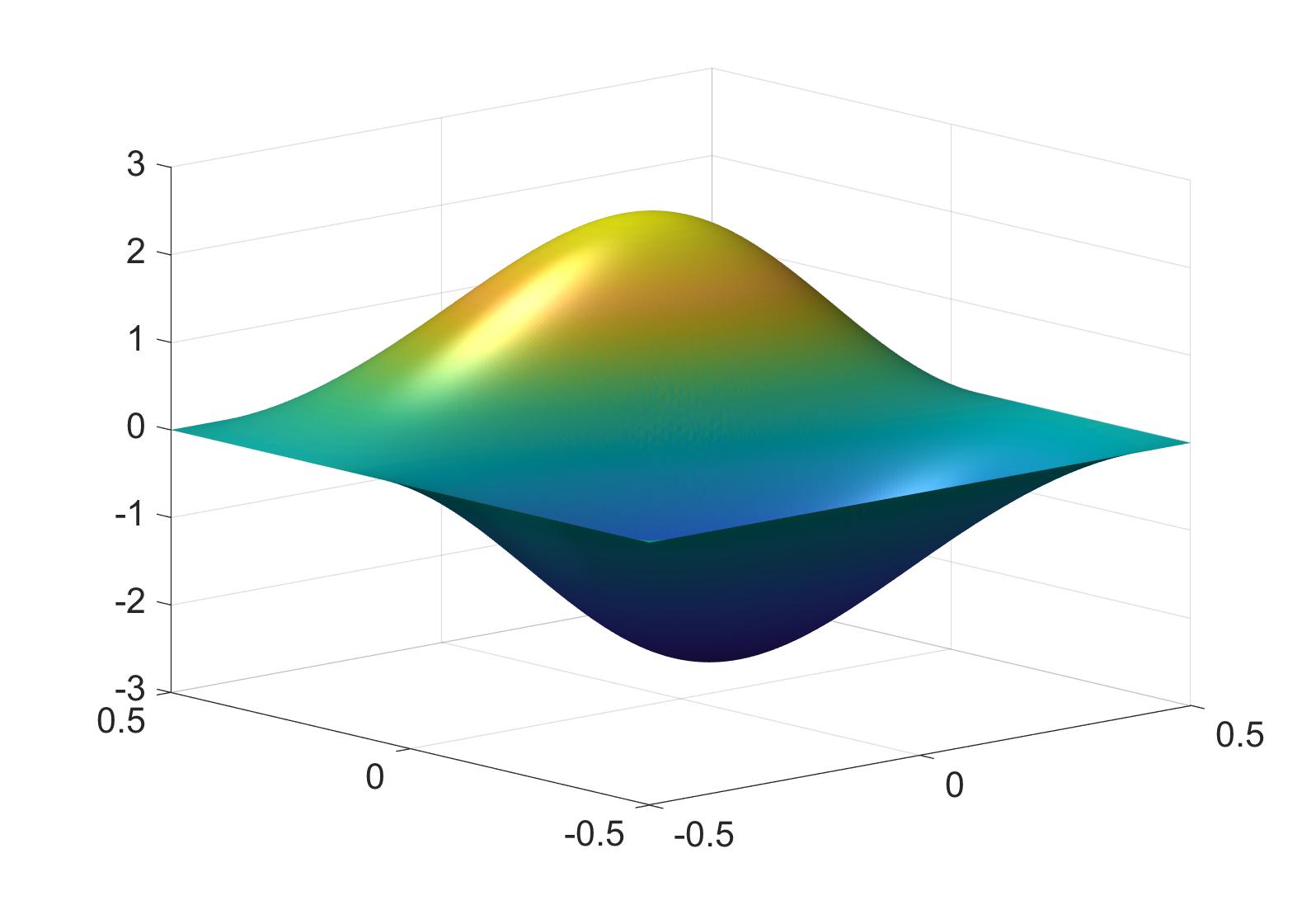}
	\hfill
	\caption{Experiment~\ref{sec:singular} (b). Left: Residual decay for the eigenstate approximation of $\psi_3$. Right: Approximated eigenstate $\psi_3$.}
	\label{fig:Singular3}
\end{figure}

\begin{figure}
	\hfill
	\includegraphics[width=0.49\textwidth]{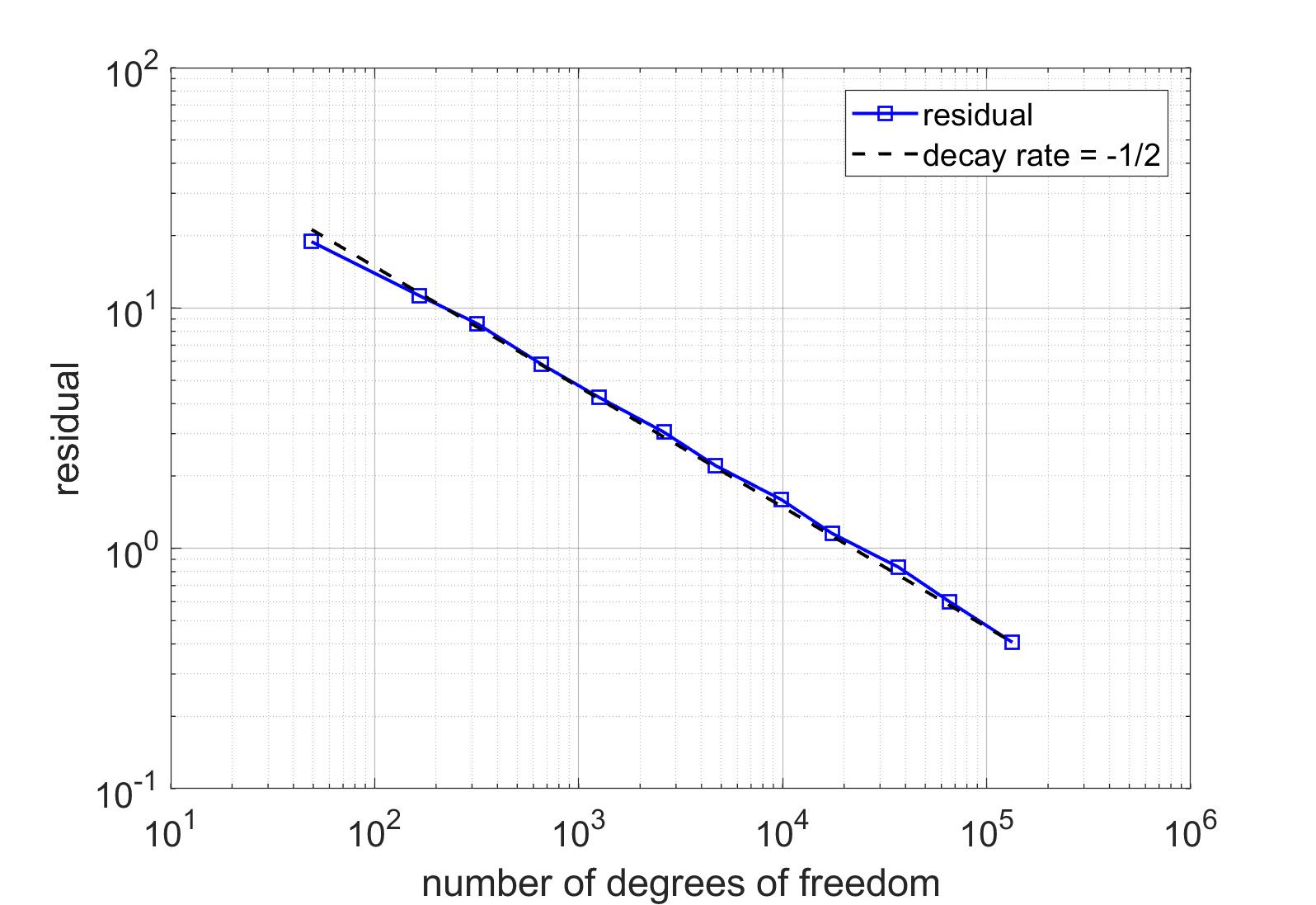}
	\hfill
	\includegraphics[width=0.49\textwidth]{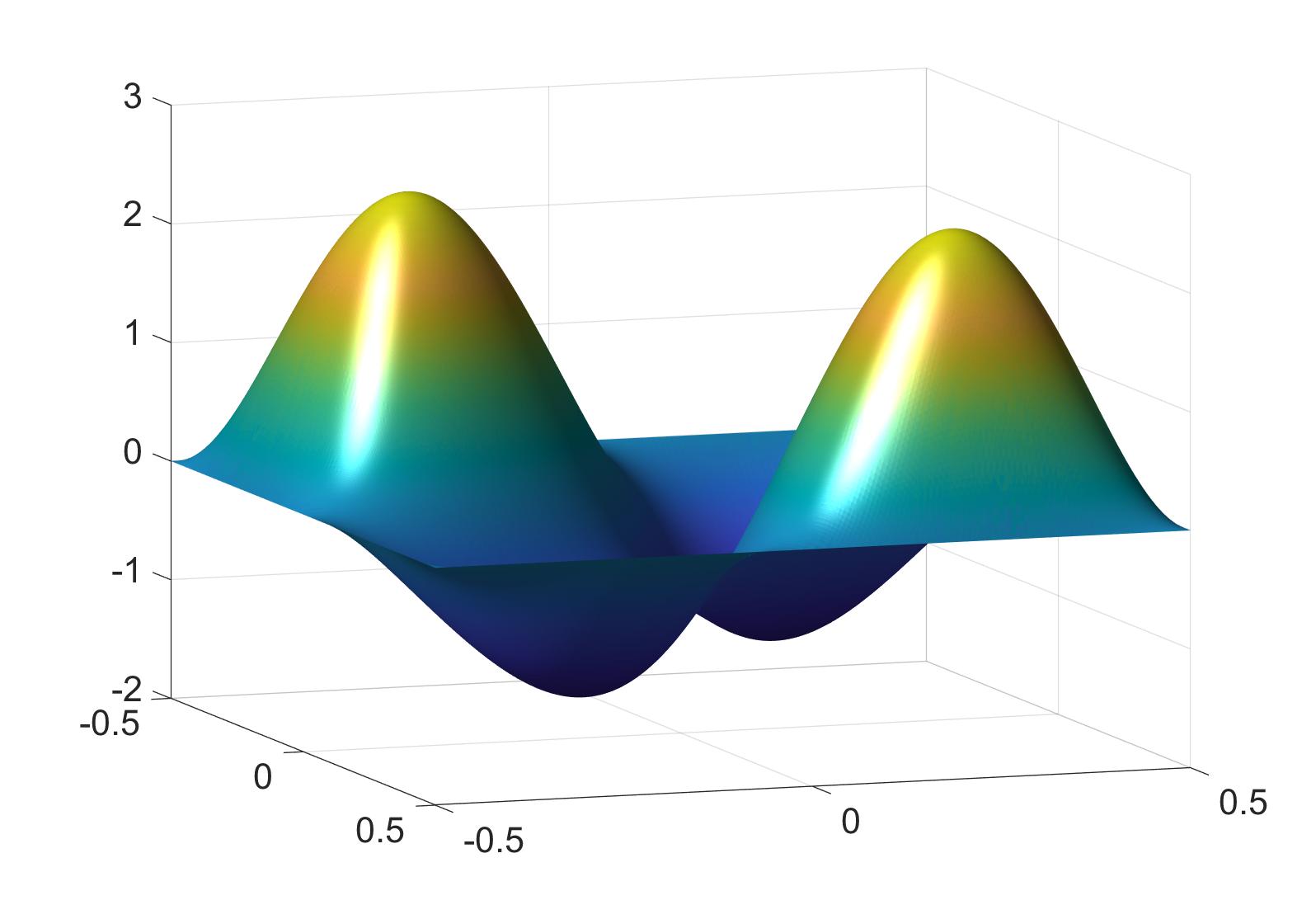}
	\hfill
	\caption{Experiment~\ref{sec:singular} (c). Left: Residual decay for the eigenstate approximation of $\psi_4$. Right: Approximated eigenstate $\psi_4$.}
	\label{fig:Singular4}
\end{figure}

\section{Conclusion} \label{sec:conclusion}

In this work we demonstrated that the computational scheme from~\cite{HeidStammWihler:19}, with some minor modifications, does well apply for the numerical approximation of excited states of Schr\"{o}dingers equation. The algorithm employs an adaptive interplay of gradient flow iterations and finite element discretisations, both of which rely on an energy minimisation. Even though the generated sequence theoretically satisfies the  orthogonality constraint in each step, a proper projection should be applied in praxis. The numerical tests illustrate that the adaptive algorithm exhibits either optimal or close to optimal convergence rates by properly resolving local features of the eigenstates.

\bibliographystyle{amsplain}
\bibliography{references,GPrefs}

\providecommand{\bysame}{\leavevmode\hbox to3em{\hrulefill}\thinspace}
\providecommand{\MR}{\relax\ifhmode\unskip\space\fi MR }
\providecommand{\MRhref}[2]{%
  \href{http://www.ams.org/mathscinet-getitem?mr=#1}{#2}
}
\providecommand{\href}[2]{#2}
\begin{thebibliography}{10}

\bibitem{AmreinWihler:14}
M.~Amrein and T.~P. Wihler, \emph{{An adaptive {Newton}-method based on a
  dynamical systems approach}}, Commun. Nonlinear Sci. Numer. Simul.
  \textbf{19} (2014), no.~9, 2958--2973.

\bibitem{AmreinWihler:15}
\bysame, \emph{Fully adaptive {N}ewton-{G}alerkin methods for semilinear
  elliptic partial differential equations}, SIAM J. Sci. Comput. \textbf{37}
  (2015), no.~4, A1637--A1657.

\bibitem{antoine2017efficient}
X.~Antoine, A.~Levitt, and Q.~Tang, \emph{{Efficient spectral computation of
  the stationary states of rotating Bose--Einstein condensates by
  preconditioned nonlinear conjugate gradient methods}}, Journal of
  Computational Physics \textbf{343} (2017), 92--109.

\bibitem{BankShermanWeiser:83}
R.~E. Bank, A.~H. Sherman, and A.~Weiser, \emph{Refinement algorithms and data
  structures for regular local mesh refinement}, Scientific computing
  ({M}ontreal, {Q}ue., 1982), IMACS Trans. Sci. Comput., I, IMACS, New
  Brunswick, NJ, 1983, pp.~3--17.

\bibitem{BaoCai:2013}
W.~Bao and Y.~Cai, \emph{Mathematical theory and numerical methods for
  {B}ose-{E}instein condensation}, Kinet. Relat. Models \textbf{6} (2013),
  1--135.

\bibitem{bao2010efficient}
W.~Bao, Y.~Cai, and H.~Wang, \emph{{Efficient numerical methods for computing
  ground states and dynamics of dipolar Bose--Einstein condensates}}, Journal
  of Computational Physics \textbf{229} (2010), no.~20, 7874--7892.

\bibitem{bao2006efficient}
W.~Bao, I-L. Chern, and F.Y. Lim, \emph{{Efficient and spectrally accurate
  numerical methods for computing ground and first excited states in
  Bose--Einstein condensates}}, Journal of Computational Physics \textbf{219}
  (2006), no.~2, 836--854.

\bibitem{BaoDu:04}
W.~Bao and Q.~Du, \emph{{Computing the ground state solution of Bose--Einstein
  condensates by a normalized gradient flow}}, SIAM Journal on Scientific
  Computing \textbf{25} (2004), no.~5, 1674--1697.

\bibitem{bao2003numerical}
W.~Bao, D.~Jaksch, and P.~Markowich, \emph{{Numerical solution of the
  Gross--Pitaevskii equation for Bose--Einstein condensation}}, Journal of
  Computational Physics \textbf{187} (2003), no.~1, 318--342.

\bibitem{bao2003ground}
W.~Bao and W.~Tang, \emph{{Ground-state solution of Bose--Einstein condensate
  by directly minimizing the energy functional}}, Journal of Computational
  Physics \textbf{187} (2003), no.~1, 230--254.

\bibitem{BernardiDakroubMansourSayah:15}
C.~Bernardi, J.~Dakroub, G.~Mansour, and T.~Sayah, \emph{{A posteriori analysis
  of iterative algorithms for a nonlinear problem}}, J. Sci. Comput.
  \textbf{65} (2015), no.~2, 672--697.

\bibitem{CancesChakirMaday:10}
E.~Canc\`es, R.~Chakir, and Y.~Maday, \emph{Numerical analysis of nonlinear
  eigenvalue problems}, J. Sci. Comput. \textbf{45} (2010), no.~1-3, 90--117.

\bibitem{chien2008two}
C-S. Chien, H-T. Huang, B-W. Jeng, and Z-C. Li, \emph{{Two-grid discretization
  schemes for nonlinear Schr{\"o}dinger equations}}, Journal of Computational
  and Applied Mathematics \textbf{214} (2008), no.~2, 549--571.

\bibitem{CongreveWihler:17}
S.~Congreve and T.~P. Wihler, \emph{Iterative {G}alerkin discretizations for
  strongly monotone problems}, Journal of Computational and Applied Mathematics
  \textbf{311} (2017), 457--472.

\bibitem{danaila2010new}
I.~Danaila and P.~Kazemi, \emph{{A new Sobolev gradient method for direct
  minimization of the Gross--Pitaevskii energy with rotation}}, SIAM Journal on
  Scientific Computing \textbf{32} (2010), no.~5, 2447--2467.

\bibitem{El-AlaouiErnVohralik:11}
L.~El~Alaoui, A.~Ern, and M.~Vohral{\'\i}k, \emph{{Guaranteed and robust a
  posteriori error estimates and balancing discretization and linearization
  errors for monotone nonlinear problems}}, Comput. Methods Appl. Mech. Engrg.
  \textbf{200} (2011), no.~37-40, 2782--2795.

\bibitem{ErnVohralik:13}
A.~Ern and M.~Vohral{\'\i}k, \emph{{Adaptive inexact {Newton} methods with a
  posteriori stopping criteria for nonlinear diffusion {PDEs}}}, SIAM J. Sci.
  Comput. \textbf{35} (2013), no.~4, A1761--A1791.

\bibitem{feppon:19}
Florian Feppon, Gr{\'e}goire Allaire, and Charles Dapogny, \emph{{Null space
  gradient flows for constrained optimization with applications to shape
  optimization}}, working paper or preprint, January 2019.

\bibitem{FoxMoler:67}
L.~Fox, P.~Henrici, and C.~Moler, \emph{Approximations and bounds for
  eigenvalues of elliptic operators}, SIAM Journal on Numerical Analysis
  \textbf{4} (1967), no.~1, 89--102.

\bibitem{GarauMorinZuppa:11}
E.~M. Garau, P.~Morin, and C.~Zuppa, \emph{{Convergence of an adaptive {Ka{\v
  c}anov} {FEM} for quasi-linear problems}}, Appl. Numer. Math. \textbf{61}
  (2011), no.~4, 512--529.

\bibitem{gong2008finite}
X-G. Gong, L.~Shen, D.~Zhang, and A.~Zhou, \emph{Finite element approximations
  for {S}chr{\"o}dinger equations with applications to electronic structure
  computations}, Journal of Computational Mathematics (2008), 310--323.

\bibitem{HeidPraetoriusWihler:2020}
P.~Heid, D.~Praetorius, and T.~P. Wihler, \emph{A note on energy contraction
  and optimal convergence of adaptive iterative linearized finite element
  methods}, Tech. Report 2007.10750, Preprint, 2020.

\bibitem{HeidStammWihler:19}
P.~Heid, B.~Stamm, and T.P. Wihler, \emph{Gradient flow finite element
  discretizations with energy-based adaptivity for the gross-pitaevskii
  equation}, Tech. Report 1906.06954, arxiv.org, 2019.

\bibitem{HeidWihler:19v2}
P.~Heid and T.~P. Wihler, \emph{Adaptive iterative linearization {G}alerkin
  methods for nonlinear problems}, Math. Comp. (2020), in press.

\bibitem{HeidWihler:20}
P.~Heid and T.P. Wihler, \emph{Adaptive local minimax galerkin methods for
  variational problems}, Tech. Report 2002.06915, Preprint, 2020.

\bibitem{HeidWihler2:19v1}
\bysame, \emph{On the convergence of adaptive iterative linearized {G}alerkin
  methods}, Calcolo \textbf{57} (2020), no.~3, 24. \MR{4131951}

\bibitem{HenningPeterseim:18}
P.~Henning and D.~Peterseim, \emph{{Sobolev gradient flow for the
  Gross-Pitaevskii eigenvalue problem: global convergence and computational
  efficiency}}, Tech. Report 1812.00835, arxiv.org, 2018.

\bibitem{HoustonWihler:16}
P.~Houston and T.~P. Wihler, \emph{Adaptive energy minimisation for $hp$-finite
  element methods}, Comput. Math. Appl. \textbf{71} (2016), no.~4, 977 -- 990.

\bibitem{HoustonWihler:18}
\bysame, \emph{An $hp$-adaptive newton-discontinuous-galerkin finite element
  approach for semilinear elliptic boundary value problems}, Math. Comp.
  \textbf{87} (2018), no.~314, 2641--2674.

\bibitem{kazemi2010minimizing}
P.~Kazemi and M.~Eckart, \emph{{Minimizing the Gross-Pitaevskii energy
  functional with the Sobolev gradient - Analytical and numerical results}},
  International Journal of Computational Methods \textbf{7} (2010), no.~03,
  453--475.

\bibitem{LinLu:2019}
L.~Lin and J.~Lu, \emph{A mathematical introduction to electronic structure
  theory}, SIAM Spotlights, vol.~4, Society for Industrial and Applied
  Mathematics (SIAM), Philadelphia, PA, 2019.

\bibitem{LinStamm:17}
L.~Lin and B.~Stamm, \emph{{A posteriori error estimates for discontinuous
  Galerkin methods using non-polynomial basis functions. Part II: Eigenvalue
  problems}}, ESAIM: M2AN \textbf{51} (2017), no.~5, 1733--1753.

\bibitem{raza2009energy}
N.~Raza, S.~Sial, S.S. Siddiqi, and T.~Lookman, \emph{Energy minimization
  related to the nonlinear {S}chr{\"o}dinger equation}, Journal of
  Computational Physics \textbf{228} (2009), no.~7, 2572--2577.

\bibitem{Sakha:2020}
I.~Sakho, \emph{Introduction to quantum mechanics 2}, John Wiley \& Sons, Ltd,
  2020.

\bibitem{Tanabe:1980}
K.~Tanabe, \emph{A geometric method in nonlinear programming}, Journal of
  Optimization Theory and Applications \textbf{30} (1980), no.~2, 181--210.

\bibitem{Verfurth:13}
R.~Verf\"urth, \emph{A posteriori error estimation techniques for finite
  element methods}, Numerical Mathematics and Scientific Computation, Oxford
  University Press, Oxford, 2013.

\bibitem{xie2016multigrid}
H.~Xie and M.~Xie, \emph{{A multigrid method for ground state solution of
  Bose-Einstein condensates}}, Communications in Computational Physics
  \textbf{19} (2016), no.~3, 648--662.

\bibitem{zeng2009efficiently}
R.~Zeng and Y.~Zhang, \emph{{Efficiently computing vortex lattices in rapid
  rotating Bose--Einstein condensates}}, Computer Physics Communications
  \textbf{180} (2009), no.~6, 854--860.

\end{thebibliography}

\end{document}